\documentclass[onecolumn]{IEEEtran}
\usepackage{graphicx} 
\usepackage{subcaption}
\usepackage[utf8]{inputenc}
\usepackage{amsmath,amsfonts,amssymb,amsthm,epsfig,epstopdf,titling}
\usepackage{url,array,mathrsfs,mathtools,dsfont,xargs}
\usepackage{stmaryrd}
\usepackage{bbold,xcolor}
\usepackage{graphicx}
\usepackage{scalerel}
\usepackage{hyperref}
\usepackage{geometry}
\usepackage{enumitem}
\usepackage{multicol}
\usepackage[authoryear]{natbib}

\theoremstyle{plain}
\newtheorem{thrm}{Theorem}
\newtheorem{prop}{Proposition}
\newtheorem{lemma}{Lemma}
\newtheorem{coro}{Corollary}

\theoremstyle{definition}

\theoremstyle{remark}

\newcommand\Rs{\mathbb R_+^*}
\newcommand\R{\mathbb R}
\newcommand\Rd{\mathbb R^{d}}
\newcommand\N{\mathbb N}
\newcommand\X{\mathbb X}
\newcommand\Sp{\mathbb S_{d-1}}
\newcommand\Pb{\mathbb P}
\newcommand\E{\mathbb E}

\newcommand\Var{\operatorname{Var}}

\newcommand\Card{\operatorname{Card}}
\newcommand\Conv{\operatorname{Conv}}

\newcommand\argmin{\operatorname{argmin}}
\DeclarePairedDelimiter\floor\lfloor\rfloor
\DeclarePairedDelimiter\ceil\lceil\rceil

\newcommand\s{\sigma}
\newcommand\la{\lambda}
\newcommand\e{\varepsilon}
\newcommand\gres{\gamma_{\textup{res}}}
\newcommand\gcov{\gamma_{\textup{cov}}}
\newcommand\gbias{\gamma_{\textup{bias}}}
\newcommand\Cvar{\mathcal C_{\textup{var}}}
\newcommand\Cres{\mathcal C_{\textup{res}}}
\newcommand\Ccov{\mathcal C_{\textup{cov}}}
\newcommand\Cbias{\mathcal C_{\textup{bias}}}
\newcommand\Cbound{\mathcal C_{\textup{boundary}}}
\newcommand\Rpq{\mathcal R_{P,Q}}

\newcommand\md{\,\mathrm d}
\newcommand\1[1]{\mathbf{1}\left(#1\right)}
\newcommand\norm[1]{\lVert #1 \rVert}
\newcommand\va[1]{\left\lvert #1 \right\rvert}
\newcommand\sva[1]{\lvert #1 \rvert}
\newcommand\vab{\sva{V^d}}
\newcommand\meanm{\frac 1m\sum_{j=1}^m}
\newcommand\meann{\frac 1n\sum_{i=1}^n}
\newcommand\meank{\frac 1k\sum_{i\in I_k(z)}}

\newcommandx\tak{\hat{\tau}_{k+1}(z)}

\newcommandx\ik[1][1=\ell]{\hat i_{#1}(z)}
\newcommandx\iky[1][1=\ell']{\hat i_{#1}(y)}

\newcommandx\ieu[2][1=n, 2=1]{\llbracket #2; #1 \rrbracket}

\newif\ifprintnotes

\begin{document}

\title{Nearest-neighbour matching on unbounded supports and covariate shift transfer}
\author{Simon VIEL}

\maketitle

\begin{abstract} 
\noindent 

Expectations of multivariate functions with missing labels occur in various fields such as transfer learning and average treatment
effects. Although non-parametric estimators based on nearest-neighbour matching are frequently used in this context, the existing literature 
assumes that the covariates live in some well-shaped compact subset of $\R^d$, with densities that are bounded away from zero. In this 
paper, we show that the usual rates of convergence can be achieved with minimal assumptions on the covariate supports. These assumptions 
are replaced with conditions on the source and target distributions, among which a measure of the tranferability between the two probability 
measures. We show that these conditions are general, can be applied to distributions supported on manifolds, and allow the target distribution 
to have a heavier tail than the source distribution. We also show that this control of the transferability is needed for any estimator to 
achieve good rates of convergence. Finally, applying our results to the estimation of treatment effects, we could relax the assumption that
the assignment probabilities had to be bounded away from zero and one.
\end{abstract}

\section{Introduction}\label{intro}

We study non-parametric estimation of an expectation of a function with missing observations. In many practical situations, we can find 
expectations involving a pair of random variables, a covariate vector and a label, for which the labels are partially observed. In the 
context called \emph{covariate shift} \citep{shimodaira2000improving}, such expectations can be encountered in various fields, such as 
econometrics \citep{heckman1979sample}, natural language processing \citep{bickel2006dirichlet}; \citep{jiang2007instance}, and 
brain-computer interfaces \citep{sugiyama2007covariate}; \citep{li2010application}. A common problem in these situations is to estimate 
an expectation of the form $\E[h(X^*,Y^*)]$, observing a sample of the covariate vector $X^*$ without the labels $Y^*$ and a sample of 
the labelled pair $(X,Y)$. In the covariate shift scenario, we assume that the conditional distributions of $Y^*$ given $X^*$ and $Y$ 
given $X$ are the same while the covariate distributions $Q:=P_{X^*}$ and $P:=P_X$ of $X^*$ and $X$ may differ. In supervised learning,
setting $h(x,y):=l(F(x),y)$ for some loss function $l$ and some hypothesis function $F$, our problem becomes risk estimation, which
is a key step for empirical risk minimisation. 

Another important situation is the estimation of treatment effects in econometrics. This context features a triplet of random variables 
$(X,Y,W)$, where $W\in\{0;1\}$ designates whether an individual received an active or a passive treatment. Of the two potential 
outcomes $Y(0)$ and $Y(1)$, only $Y_i=Y_i(W_i)$ is observed for the individual $i$. The objective of estimating the average treatment effect 
$\E[Y(1)-Y(0)]$, using the observations $(X_i,Y_i,W_i)_{i=1,\dots,N}$, then relates to our problem.

In the case where the two covariate distributions $Q$ and $P$ have respective densities $q$ and $p$ with respect to a reference measure, the
quantity of interest can be rewritten as an expectation weighted by a density ratio: $\E[h(X^*,Y^*)]=\E[h(X,Y)q(X)/p(X)]$. In that 
regard, many sources have introduced weights for the source observations that estimate the density ratio $q/p$, also called the 
\emph{importance}, see \citep{sugiyama2007direct}, \citep{kanamori2009least}, \citep{loog2012nearest}, and \citep{zhang2020one}.

Another natural approach to covariate shift adaptation is to consider, as an estimator of the expectation $\E[h(X^*,Y^*)]$, an empirical 
average over the target sample $(X_j^*)_{j=1,\dots,m}$ such as $\meanm\hat h(X_j^*)$, where $\hat h(X_j^*)$ replaces $h(X_j^*,Y_j^*)$ due
to the missing label $Y_j^*$. Therefore, such an estimator only depends on the pointwise estimator $\hat h(z)$. However, because we are 
aiming to estimate an expectation and take the average over the target sample of the pointwise estimator $\hat h(z)$, the latter need 
not be consistent for the final estimator to be consistent. Instead, it turns out to be more beneficial to use a very localised pointwise 
estimator for $\hat h(z)$. Such local estimators can be obtained, for example, using $k$-nearest neighbours ($k$-NN) algorithms with a 
small parameter $k$. It is known that while the variance of a $k$-NN pointwise predictor scales in $1/k$, the variance of the final 
estimator of the expectation scales in $1/n$ for all choices of $k$, where $n$ is the sample size of the source sample. Because the bias 
of $k$-NN estimators increases with $k$, taking $k=1$ is then a common choice. In a series of studies on the estimation of average 
treatment effects (ATE), a problem related to ours, Abadie and Imbens (\citeyear{abadie2006large}, \citeyear{abadie2008failure},
\citeyear{abadie2011bias}, \citeyear{abadie2012martingale}) established several asymptotic results for $k$-NN matching estimators with
fixed $k$. The link between matching estimators and importance estimation was shown in \citet{lin2023estimation}. In the context of 
covariate shift adaptation, $1$-NN algorithms can be found in \citet{loog2012nearest} and \citet{portier2024scalable}. Recently, $k$-NN 
matching estimators were extended by \citet{holzmann2026multivariate} to $k$-NN local polynomial estimators to reduce the bias in higher 
dimensions. Other estimation problems with expectations or integrals have also used $1$-NN algorithms, for instance to estimate entropy 
\citep{sricharan2012estimation} or residual variance \citep{devroye2018nearest}. Algorithms based on $k$-NN matching have fast computation
times and require no hyper-parameter tuning. As we shall see, they also perfectly take into account the decay of the source probability
mass.

However, the study of $k$-NN type estimators in the literature is frequently connected to the \emph{strong density assumption}
\citep{audibert2007fast}, assuming that the covariates have a density that is bounded away from zero and infinity on some compact 
well-shaped subset of $\R^d$. This assumption can be found in \citet{abadie2006large}, \citet{lin2023estimation}, 
\citet{portier2024scalable}, \citet{holzmann2026multivariate}, and \citet{viel2025convergence}. Although this assumption can be a standard
approximation in practice, the way rates of convergence for the estimators depend on the approximated support diameter or density bounds,
for instance, was often not made explicit in previous studies. In this paper, we show that the same rates of convergence can be obtained
under milder assumptions on the covariate supports and distributions. We extend previous results to unbounded domains, and we also allow
probability densities to go to zero near the boundary of the support. In contrast, we introduce conditions that measure the transferability
between the two probability measures $P$ and $Q$, based on the integrability of functions linked to the importance $\md Q/\md P$. 
Studying the moments of the density ratio is a natural way to measure transferability between the source and target distributions 
\citep{ma2023optimally}. Similar measures were obtained on bounded supports using ball-mass ratios, where the singularity between the 
source and target distributions is quantified via the \emph{transfer exponent} \citep{kpotufe2021marginal}; \citep{suk2021self};
\citep{cai2024transfer}, or a variant using $\alpha$-families of distributions \citep{pathak2022new}. Eventually, the transfer function
\citep{zamolodtchikov2026minimax} measures the transferability on unbounded domains using local mass assumptions on the two probability
distributions. In this paper, we use a related transferability measure that can be applied to wider classes of multivariate distributions,
including Gaussian distributions. For our specific problem, we even allow the target distribution to have a heavier tail than the source
distribution.

In a second direction, we show that if the covariate observations $X$ and $X^*$ are given as vectors in some Euclidean space $\R^d$, the 
associated distributions $P$ and $Q$ need not have density with respect to the Lebesgue measure of $\R^d$. Instead, the relevant 
information lies in the \emph{intrinsic dimensions} of $P$ and $Q$, characterised by the volumes of balls under these measures. It means 
that the observations may be located on a $d_0$-dimensional submanifold of $\R^d$, up to some small noise. We stress that no knowledge
about the supports or the intrinsic dimensions is required in order to compute the estimators. Nonetheless, the size-dependent rates of 
convergence that we obtain depend only on the intrinsic dimensions of $P$ and $Q$ and not on the dimension $d$ of the overlying space.

Let $d$ be a positive integer, $\mathcal Y$ a measurable space, $h:\Rd\times\mathcal Y\to\R$ a given measurable function, and let $P_{
X,Y}$ and $P_{X^*,Y^*}$ be two probability measures on $\Rd\times\mathcal Y$. We intend to estimate the expectation $$e(h):=\E[h(X^*,
Y^*)]=\int_{\Rd\times\mathcal Y}h(x,y)\md P_{X^*,Y^*},$$ using only an unlabelled sample $(X_j^*)_{j=1,\dots,m}$ from $P_{X^*}$ and an
independent labelled sample $(X_i,Y_i)_{i=1,\dots,n}$ from $P_{X,Y}$. This formulation enables us to tackle regression ($\mathcal Y=\R$) 
and classification ($\mathcal Y=\{-1;1\}$) problems at the same time. The \emph{source} distribution $P:=P_X$ and the \emph{target} distribution 
$Q:=P_{X^*}$ are assumed to have support in some (possibly unknown) Borel set $\X\subset\R^d$. We can then equivalently write $$e(h)=\E[
g_h(X^*)]=\int_\X g_h(z)\md Q(z),$$ where $g_h:z\in\X\mapsto\E[h(X^*,Y^*)\mid X^*=z]\in\R$ is the regression function. Note that by the 
covariate shift assumption, the conditional distributions of $Y$ given $X$ and of $Y^*$ given $X^*$ are the same, so we also have
$g_h(z)=\E[h(X,Y)\mid X=z]$. When there is no ambiguity, we shall simply write the function $g$ instead of $g_h$. The final estimators 
we consider have the form $$\hat e(h):=\hat e_{m,n}\big(\hat h\big):=\meanm \hat h(X_j^*),$$ where $\hat h$ is a random function independent 
from the target sample $(X_j^*)_{j=1,\dots,m}$. When the pointwise estimator is linear in the observed labels, that is when $\hat h(z)=
\sum_{i=1}^nw_i(z)h(X_i,Y_i)$ for some random $(X_1,\dots,X_n)$-measurable weights $w_1(z),\dots,w_n(z)\in\R$, then the final
estimator can be written as the weighted average $$\hat e(h)=\sum_{i=1}^n\left(\meanm w_i(X_j^*)\right)h(X_i,Y_i).$$ 
We consider non-parametric pointwise estimators $\hat h(z)$ that are based on nearest-neighbour algorithms. More precisely, for a fixed 
nonnegative integer $L$, the pointwise estimator $\hat h_L(z)$ approximates the regression function $g_h$ by a polynomial of order $L$
inside the $k$th-order Voronoï cell created around the point $z$ with the sample $(X_i)_{i=1,\dots,n}$. The coefficients of the polynomial
are then fitted by a least-square minimisation using the source observations $(X_i,Y_i)$ such that $X_i$ belongs to this Voronoï cell. The
case where $L=0$ is of special interest, where we approximate the function $g_h$ in the Voronoï cell by a constant.


\bigskip

To summarise, the key contributions of this paper are the following.
\begin{enumerate}
\item We show that the usual rates of convergence for $k$-NN matching and local polynomial estimators extend to situations with general
unbounded supports, where the covariate densities are allowed to approach zero and infinity.
\item We relate these rates to the intrinsic dimensions of the source and target distributions, in the case the covariate vectors are
supported on lower dimensional subspaces, up to some small noise.
\item We discuss the general conditions we introduce on the relative behaviour of the source and target distributions.
\item We apply our results to the estimation of treatment effects under the same general assumptions for the covariate vector, and we also 
allow the probability to assign a given treatment to approach zero.
\end{enumerate}

This paper is organised as follows. Section \ref{simple} contains the results for the rates of the matching and the local polynomial
estimators on an unbounded domain, under simple design assumptions that show the effect of the covariate shift transferability, followed by a
discussion of minimax lower bounds. Section \ref{general} contains extensions of these results for both estimators under general design
assumptions. An application of these results to the estimation of average treatment effects can be found in Section \ref{ate}. In Section
\ref{ratio}, the conditions on the source and target distributions are shown to be general and easy to interpret, even when the
distributions are supported on manifolds. Possible extensions of the results are discussed in Section \ref{extensions}. Numerical 
experiments are given in Section \ref{experiments}. A conclusion is given in Section \ref{conclusion} and proofs of our results can be
found in Section \ref{proofs}. Technical lemmas are gathered in an Appendix Section.

\bigskip

\section{Rates for the estimators in a special case}\label{simple}

\subsection{Notation}\label{notation}

We denote by $\N^*$ the set of positive integers and by $\Rs$ the set of positive real numbers. Let $d\in\N^*$, the Euclidean space of 
dimension $d$ is written $\Rd$. We denote by $\norm\cdot$ an arbitrary norm on $\Rd$ and by $B(z,r)$ the closed ball in $\Rd$ with center 
$z\in\Rd$ and radius $r\in\Rs$. The derivative of order $\la\in\N^*$ of a function $f$ defined on a subset of $\Rd$ is written $D^\la f$, 
with $D^1f$ being identified to the gradient $\nabla f$. We let $\norm f_{\infty,A}$ denote the supremum of the norm of $f$ on the set 
$A\subset\Rd$. The indicator function $\1 A$ is equal to $1$ on $A$ and $0$ on the complement of $A$. The set $\Conv(A)$ is the convex 
hull of $A$ in $\Rd$. For $a,b\in\R$, $a\wedge b$ and $a\vee b$ denote, respectively, the minimum and the maximum between $a$ and $b$. 
Finally, for $\ell\in\R,\,\ceil\ell$ designates the smallest integer greater than or equal to $\ell$, while $\floor\ell$ designates the largest 
integer smaller than or equal to $\ell$.

For $n\in\N^*$, we denote by $\ieu$ the set of integers $i$ such that $1\le i\le n$. For all $k\in\ieu[n+1]$ and $Q$-almost all $z\in\X$, 
we set $\hat\tau_k(z)$ the distance between $z$ and its $k$th-nearest neighbour among $(X_i)_{i\in\ieu}$, with the convention 
$\hat\tau_{n+1}(z):=\infty$, and we set $I_k(z)\subset\ieu$ the set of indexes of the $k$-nearest neighbours to $z$, that is,
$$I_k(z)=\{i\in\ieu\mid\norm{X_i-z}<\tak\}.$$ To avoid possible ties, we assume that for $Q$-almost all $z\in\X$, the distribution 
$P$ does not charge the spheres of center $z$ and with radius $r\in[0,r_0]$, where $r_0\in\Rs$ is fixed. Under this assumption, 
we know that, almost-surely, for $Q$-almost all $z\in\X,\,I_k(z)$ has cardinal $k$. Note that this assumption is not very restrictive 
even for distributions supported on submanifolds.

\subsection{Rates for the matching estimator on an unbounded domain}\label{match_simple}

In this section, we provide a direct extension of the results on the convergence rates for the matching estimator when the covariate vector 
has a density that is allowed to approach zero. Let $k\in\ieu$ and $r_0\in\Rs$ be fixed. We consider the final estimator $\hat e_0(h)=
\hat e_{m,n}\big(\hat h_0\big)$, with the choice of the classical pointwise matching estimator, defined for all $z\in\X$ by 
$$\hat h_0(z):=\meank h(X_i,Y_i)\text{ if }\tak\le r_0,\text{ else }\hat h_0(z):=0.$$

Censoring when the NN-radius is too large is in general necessary with unbounded supports to avoid high-distance integrability 
difficulties. It is also justified by the fact that, for a general smooth function $g$, no relevant information about the quantity $g(z)$
can be deduced from the observations when they are all far from the point $z$. The exact value of the censor radius $r_0$ only affects
the constants in the rates of convergence of the estimators, and its tuning is beyond the scope of this paper. We can therefore consider
$r_0=1$, which generally has little impact on the conditions below where the radius $r_0$ is used. It is also possible to omit censoring under
additional assumptions on the regression function $g$, for example, if $g$ is bounded $Q$-almost surely on $\R^d$. In this case, any unknown 
radius $r_0$ can be used in the conditions below.

\bigskip

We now state conditions on the distributions of the random variables $(X,Y)$ and $X^*$. In this section, we consider a simple design 
that makes the conditions on the relative behaviour of the source and target distributions appear clearly. A more general setup is
described in Section \ref{general}.

\begin{enumerate}[label=(X\arabic*), wide=0.5em, leftmargin=*]
\item\label{cond:X1} The conditional variance $\Var(h(X,Y)\mid X)$ is bounded by some constant $\s^2\in\Rs$. There exists $\ell\in(0,1]$
such that the regression function $g$ is bounded and $\ell$-H\"older continuous on $\Rd$.
\item\label{cond:X2} The distributions $P$ and $Q$ are absolutely continuous with respect to the Lebesgue measure on $\Rd$, with
respective densities $f_P$ and $f_Q$. Besides, there exist constants $c_P,C_Q\in\Rs$ such that for all $z\in\X$ and all
$r\in[0,2r_0]$, $P(B(z,r))\ge c_Pf_P(z)r^d$, and $Q(B(z,r))\le C_Qf_Q(z)r^d$.
\item\label{cond:X3} The following integrals are finite: $$\int_\X\frac{f_Q(z)^2}{f_P(z)}\md z<\infty,\text{ and}\quad
\int_\X\frac{f_Q(z)}{f_P(z)^{1/2}}\md z<\infty.$$
\end{enumerate}

Condition \ref{cond:X1} is standard in the non-parametric regression literature. Condition \ref{cond:X2} allows us to control the
volume of a ball under $P$ and $Q$ by the volume of the ball for the Lebesgue measure times the density of $P$ or $Q$ taken at the
center of the ball. This condition, which will be relaxed in the next section, is satisfied for several classes of distributions
on unbounded domains, such as exponential distributions or Pareto distributions, and it has been used in previous studies on covariate shift 
learning \citep{zamolodtchikov2024transfer}; \citep{zamolodtchikov2026minimax}. Finally, Condition \ref{cond:X3} features the
\emph{covariate shift transferability} condition. The first integral condition is equivalent to asking the importance function 
$\mathrm dQ/\mathrm dP=f_Q/f_P$ to be integrable with respect to the measure $Q$, or equivalently square-integrable with respect 
to the measure $P$. This condition was already present in \citet{ma2023optimally} or in studies on importance estimation; see
\citet{kanamori2009least}. Meanwhile, the second integral condition is very similar to the first one, and the two will be equivalent in some 
cases. Following the definitions of the transfer function and the integrability index $\gamma^*$ from \citet{zamolodtchikov2026minimax},
our condition means that $\gamma^*(P,Q)\ge 1/2$. In many situations, we can interpret this condition by saying that the target distribution 
$Q$ is allowed to have a heavier tail than the source distribution $P$, as long as it is not twice as heavy.
We are now in a position to state the following theorem.

\begin{thrm}[Mean-squared error]\label{simple_mse}
Suppose that we have Conditions \ref{cond:X1} to \ref{cond:X3}. Then there exists a constant $C\in\Rs$ such that the mean-squared
error of the estimator $\hat e_0(h)$ satisfies $$\E[(\hat e_0(h)-e(h))^2]\le C\left(\frac 1m+\frac 1n+\left(\frac kn\right)^{
1\wedge 2\ell/d}\right).$$
\end{thrm}

In Theorem \ref{simple_mse}, because the upper bound of the variance term $C(m^{-1}+n^{-1})$ does not depend on the parameter $k$, 
taking $k=1$ can be optimal for the mean-squared error as it minimises the upper bound of the bias term $C(k/n)^{1\wedge 2\ell/d}$. 
The estimator $\hat e_0(h)$ then achieves the parametric rate of convergence $m^{-1}+n^{-1}$ as soon as $d\le 2\ell$, where $\ell\in(0,1]$
is the regularity of the regression function. Therefore, it requires $d\le 2$ for a Lipschitz function $g$. In higher dimensions and 
with more regularity on the function $g$, we can give an extension to the estimator $\hat e_0(h)$ that has a lower bias than in the 
upper bound of Theorem \ref{simple_mse}. This extension is the object of the next section.

\subsection{Rates for the local polynomial estimator on an unbounded domain}\label{poly_simple}

In this subsection, we analyse the rates of convergence of the estimator $\hat e_L(h)=\hat e_{m,n}\big(\hat h_L\big)$, with the choice of 
a more general pointwise estimator $\hat h_L$, introduced by \citet{holzmann2026multivariate} to reduce the bias when the regression 
function $g$ has more regularity. They use local polynomials of order $L\in\N$ to approximate the regression function inside the Voronoï 
cells, where the choice $L=0$ leads to the classical matching estimator from the previous section.

We set $N_L$ the set of $d$-tuples of non-negative integers whose sum does not exceed $L$, with cardinality $K^*:=K^*(d,L)$. For $\alpha
\in N_L$, we write $\zeta_\alpha(z,T):=(T-z)^\alpha:=\prod_{j=1}^d(T_j-z_j)^{\alpha_j}$, which is a local monomial in $T$ in $d$
coordinates, of degree at most $L$, and centered at the point $z$. Finally, $\zeta(z,T)=(\zeta_\alpha(z,T))_{\alpha\in N_L}$ will denote
the vector in $\R^{N_L}$ made up of all such local monomials. The estimator of the regression function and its derivatives is the least 
square estimator $$\hat G_L(z):=\underset{b\in\R^{N_L}}\argmin\sum_{i\in I_k(z)}\left(h(X_i,Y_i)-b^\top\zeta(z,X_i)\right)^2.$$
Therefore, the pointwise estimator of the regression function is 
$$\hat h_L(z):=e_0^\top\hat G_L(z)=e_0^\top M(z)^{-1}\sum_{i\in I_k(z)}h(X_i,Y_i)\zeta(z,X_i)\text{ if }\tak\le r_0,\text{ else }
\hat h_L(z):=0,$$ where $e_0$ is the vector in the canonical basis of $\R^{N_L}$, corresponding to the tuple $(0,\dots,0)\in N_L$, and
$M(z)$ is the square matrix $$M(z):=\sum_{i\in I_k(z)}\zeta(z,X_i)\zeta(z,X_i)^\top=\left(\sum_{i\in I_k(z)}(X_i-z)^{\alpha+\alpha'}
\right)_{\alpha,\alpha'\in N_L}.$$
Note that when $L=0,\,\zeta(z,T)$ is the constant polynomial $1\in\R,\,M(z)=k$, and $\hat h_L(x)$ coincides with the classical estimator 
$\hat h_0(x)$ defined in Section \ref{match_simple}. In what follows, we will need the constant $D:=D(d,L):=\sum_{i=1}^Li\binom{d+i-1}i$. 
In addition to the list of conditions from Section \ref{match_simple}, we consider the following conditions.

\begin{enumerate}[label=(X1'), wide=0.5em, leftmargin=*]
\item\label{cond:X1'} The conditional variance $\Var(h(X,Y)\mid X)$ is bounded by some constant $\s^2\in\Rs$. The regression function $g$
is bounded and $L$-times differentiable on $\Rd$, and its derivative of order $L$ is a bounded and Lipschitz continuous function on $\Rd$.
\end{enumerate}
\begin{enumerate}[label=($\mathcal L$), wide=0.5em, leftmargin=*]
\item\label{cond:L} There exists a constant $V_P\in\Rs$ such that for all $z\in\X$ and all $x\in B(z,3r_0),\,f_P(x)\le V_Pf_P(z)$.
\end{enumerate}

Condition \ref{cond:X1'} asks for more regularity on the regression function $g$ than in the previous section, while Condition 
\ref{cond:L} complements Condition \ref{cond:X2} with a local upper bound on the source density. As we discuss in Section \ref{ratio}, it 
is satisfied for the same types of distributions as the ones satisfying Condition \ref{cond:X2}.
Adapting the nice proof technique of \citet{holzmann2026multivariate}, we obtain the following result. 

\begin{thrm}[Mean-squared error]\label{simple_poly_mse}
Suppose that we have Conditions \ref{cond:X1'}, \ref{cond:X2}, \ref{cond:X3}, and \ref{cond:L}. Then, taking $k\ge(2D+1)K^*$, there
exists a constant $C\in\Rs$ such that the mean-squared error of the estimator $\hat e_L(h)$ satisfies
$$\E[(\hat e_L(h)-e(h))^2]\le C\left(\frac 1m+\frac 1n+\left(\frac kn\right)^{1\wedge 2(L+1)/d}\right).$$
\end{thrm}

Theorem \ref{simple_poly_mse} shows similar rates of convergence than in Theorem \ref{simple_mse}, that is, $m^{-1}+n^{-1}+(k/n)^{1\wedge 
2\ell/d}$, where $\ell$ denotes the regularity of the regression function, but this time allowing $\ell>1$. As a consequence, the local
polynomial estimator $\hat e_L(h)$, with a constant parameter $k$, can achieve the parametric rate of convergence $m^{-1}+n^{-1}$ in
dimension $d>2$, as long as we choose $L+1\ge d/2$. In Section \ref{general}, we show that the same result holds under milder assumptions on
the regression function $g$ and the covariate distributions $P$ and $Q$. But first, we prove minimax lower bounds for the estimation of
$e(h)$ that show that the integrability of the importance function $f_Q/f_P$, as in Condition \ref{cond:X3}, affects the best achievable rate
of convergence.

\subsection{Minimax lower bounds}\label{minimax_rates}

In this section, we provide minimax lower bounds for the estimation of the expectation $e(h)$ depending on the transferability between the 
source and the target probability measures. We start by fixing the dimension $d\in\N^*$ and a regularity parameter $\ell\in\N^*$, which
can be arbitrarily large. In this study, we consider the regression setup $Y=g(X)+\e$, where $\e$ is a Gaussian random variable
independent from $X$, and we consider the function $h(x,y):=y$. To state minimax lower bounds, we introduce a class for the distributions
of $X,Y,X^*$ that satisfy a list of conditions that focus on the integrability of the function $f_Q^\alpha/f_P^\gamma$, as in Condition
\ref{cond:X3}. Note that in our regression setup, the couple of distributions $(P_{X,Y},Q)$ depends only on the covariate distributions
$P$ and $Q$, the regression function $g:\R^d\to\R$, and the distribution of the noise $\e$. For $\gamma\in[0,1]$ and $\Ccov\in\Rs$, we say 
that the pair $(P_{X,Y},Q)$ belongs to the class $\Sigma(\gamma,\Ccov)$ when:
\begin{itemize}
\item the noise $\e$ follows the Gaussian distribution $\mathcal N(0,\Lambda^2)$ for some $\Lambda^2\le 1$,
\item the regression function $g$ has derivatives up to order $\ell$ on $\R^d$, and for all $s\in\ieu[\ell][0],\,\frac{\norm{D^sg}_{\infty,
\R^d}}{s!}\le\Lambda$,
\item the distributions $P$ and $Q$ have density with respect to the Lebesgue measure on $\R^d$, with respective densities $f_P$ and 
$f_Q$ that have infinitely many derivatives, all bounded by $1$, on their convex support, and
\item we have the transferability control $$\Lambda^2\int_{\R^d}\frac{f_Q(z)^2}{f_P(z)^\gamma}\md z\le\Ccov,\text{ and}\quad\Lambda\int_{\R^d}
\frac{f_Q(z)}{f_P(z)^{\gamma/2}}\md z\le\sqrt{\Ccov}.$$
\end{itemize}

In the last point above, we relax Condition \ref{cond:X3}, which corresponds to the case where $\gamma=1$, allowing the target distribution 
to have a heavier tail than in the previous sections. However, the following theorem shows that we can then lose the parametric rate of
convergence.

\begin{thrm}\label{minimax}
There exists a universal constant $\mathcal K\in\Rs$ such that for all $\gamma\in[0,1],\,\Ccov\in\Rs$ and $n\in\N^*$ large enough,
$$\underset{\hat E_n}\inf\;\underset{(P_{X,Y},Q)\in\Sigma(\gamma,\Ccov)}\sup\E[(\hat E_n-e(h))^2]\ge\mathcal K\;\frac{\Ccov}{n^\gamma
\ln(n)^2}.$$
\end{thrm}

The above theorem shows that the control of the covariate shift transferability is a necessary ingredient to obtain good rates of 
convergence, and an arbitrary number of bounded derivatives for the regression function $g$ cannot compensate for a poor transfer 
exponent $\gamma<1$. In addition, we shall see in Section \ref{polynomial} that the $k$-NN local polynomial estimator achieves 
the rate of convergence $\Ccov\,n^{-\gamma}$ as soon as $d\gamma\le 2\ell$, which shows that in this case the minimax lower bound is 
tight up to a logarithmic factor.

\section{Rates for the estimators under general conditions}\label{general}

In this section, we extend the results from Section \ref{match_simple} under milder assumptions on the distributions of $(X,Y)$ and 
$X^*$. Because we do not assume that the covariate support $\X$ is compact, the regression function and its derivatives need not be 
bounded on the whole domain $\X$, but only locally. We also show that the covariate distributions $P$ and $Q$ need not be absolutely 
continuous with respect to the Lebesgue measure on $\Rd$. This latter assumption was only used to control the volumes of balls under 
the measures $P$ and $Q$, as in the second part of Condition \ref{cond:X2}. We now therefore work directly with conditions on the 
volumes of balls under $P$ and $Q$, which offers several benefits. First, they give definitions for the intrinsic dimensions of 
the measures $P$ and $Q$, which limit the rates of convergence of the $k$-NN estimators instead of the overlying dimension $d$. This 
allows the study of covariate distributions supported on known or unknown submanifolds of $\Rd$, without any modification to the 
estimators. Second, even when the covariate distributions have density with respect to the Lebesgue measure of $\R^d$, studying 
the volumes of balls instead of the densities works better in several cases, including Gaussian distributions, distributions with a 
density that goes to infinity near a finite point, or distributions supported on a non-regular support. These examples are addressed 
in Section \ref{ratio}.

\subsection{Rates for the local polynomial estimator}\label{polynomial}

In this section, we study the rates of convergence of the local polynomial estimator under milder assumptions than in Section 
\ref{simple}. Keeping the notation from Section \ref{poly_simple}, we recall the definition of the pointwise local polynomial 
estimator, defined for $Q$-almost all $z\in\X$ by
$$\hat h_L(z):=e_0^\top M(z)^{-1}\sum_{i\in I_k(z)}h(X_i,Y_i)\zeta(z,X_i)\text{ if }\tak\le r_0,\text{ else }\hat h_L(z):=0.$$
Recall that if we set $L=0$, then $\hat h_L$ coincides with the classical matching estimator, defined as $\hat h_0(z)=\meank h(X_i,Y_i)$ 
when $\tak\le r_0$. This means that we discuss the properties of the two estimators at the same time, $L=0$ being a particular case.
We introduce the following conditions, generalising the list from Section \ref{simple}.
\begin{enumerate}[label=(A\arabic*), wide=0.5em, leftmargin=*]
\item\label{cond:A1} There exist $\ell\in(0,L+1]$ and measurable functions $\mathcal G_\ell,\s^2:\X\to\R_+$ such that for $Q$-almost 
all $z\in\X$, the conditional variance $x\mapsto\Var(h(X,Y)\mid X=x)$ is bounded by $\s^2(z)$ on $B(z,r_0)\cap\X$, and 
the regression function $g$ is the restriction of a function having derivatives of order $\ceil\ell-1$ that are all
$(\ell+1-\ceil\ell)$-H\"older continuous on $B(z,3r_0)\cap\Conv(\X)$, with constant $\mathcal G_\ell(z)\ceil{\ell-1}!$.
\item\label{cond:A2} There exist $d_P,d_Q\in\Rs$ and measurable functions $p,\Rpq:\X\to\Rs,v:\X\times[0,r_0]\to[1,\infty[$, such that for 
$Q$-almost all $z\in\X$, all $r\in[0,r_0]$, and all Borel sets $\mathcal A\subset B(0,3)^{K^*}$, we have \begin{align*}
P(B(z,r)) & \ge p(z)r^{d_P}, \\ Q(B(z,2r)) & \le\Rpq(z)\,r^{d_Q-d_P}P(B(z,r)),\text{ and} \\ \Pb\left(\left(\frac{X_1-z}r,\dots,
\frac{X_{K^*}-z}r\right)\in\mathcal A\right) & \le(v(z,r)p(z)r^{d_P})^{K^*}\va{\mathcal A},
\end{align*}
where $\va{\mathcal A}$ designates the Lebesgue measure of the Borel set $\mathcal A\subset\R^{dK^*}$.
\item\label{cond:A3} There exist $\gres\in[0,d_Q/d_P],\,\gcov\in[0,(d_Q+2\ell)/d_P]$, and $\gbias\in[0,\ell/d_P]$, such that the following
integrals are finite: 
$$\Cvar:=\int_\X v(z)^{2DK^*}(\norm g_{\infty,B(z,r_0)\cap\X}^2+\s^2(z))\md Q(z)<\infty,$$
$$\Cres:=\int_\X\left[v(z)^{2DK^*+1-\gres}\s^2(z)\right]p(z)^{1-\gres}\Rpq(z)\md Q(z)<\infty,$$
$$\Ccov:=\int_\X\left[g(z)^2+v(z)^{2DK^*+1-\gcov}\mathcal G_\ell(z)^2\right]p(z)^{1-\gcov}\Rpq(z)\md Q(z)<\infty,\text{ and}$$ 
$$\Cbias:=\int_\X\left[\va{g(z)}+v(z)^{DK^*}\mathcal G_\ell(z)\right]p(z)^{-\gbias}\md Q(z)<\infty,$$
where we have set, for $Q$-almost all $z\in\X$, $v(z):=\underset{r\in[0,r_0]}\sup\frac{p(z)r^d}{P(B(z,r))}v(z,r)\le\underset{r\in[0,r_0]}\sup
v(z,r)$.
\end{enumerate}

\bigskip

Condition \ref{cond:A1} extends Conditions \ref{cond:X1} and \ref{cond:X1'} by assuming local bounds on the conditional variance, the 
regression function, and its derivatives. Asking $g$ to be the restriction of a regular function on the convex hull of $\X$ is a technical 
condition ensuring the validity of Taylor expansions around points of $\X$. 

Conditions \ref{cond:A2} and \ref{cond:A3} are the main subjects of Section \ref{ratio}. Defining the intrinsic dimension of a measured metric 
space by the volume of balls is classical in the literature \citep{kpotufe2021marginal}; \citep{pathak2022new}. Note that we assume that the
support $\X$ is a subset of an Euclidean space $\R^d$ instead of a general metric space only to have access to higher order Holder regularity 
and therefore Taylor expansions for the local polynomial estimators with $L\ge 1$.
Intuitively, the numbers $d_P$ and $d_Q$ in Condition \ref{cond:A2} represent the intrinsic dimensions of the distributions $P$ and $Q$,
respectively, and they limit the convergence rates of the estimators. When $P$ and $Q$ are absolutely continuous with respect to the Lebesgue 
measure of $\Rd$, we can take $d_P=d_Q=d$. In general, because Condition \ref{cond:A2} holds for $Q$-almost all $z$, we must have $d_Q\le d_P$, 
which can be expected in view of the nature of the covariate shift problem, where the measure $Q$ must be absolutely continuous with respect to 
$P$. However, we will see that having $d_Q=d_P$ is needed to achieve parametric rates of convergence. We stress that these dimensions 
need not be known to compute the estimators. 

Then the function $p$ could be up to some constant the density function of $P$, as in Condition \ref{cond:X2}, while the function $v$ 
represents the local multiplicative amplitude of $P$. We will see that the function $v$ often grows slowly compared to the vanishing of $p$, 
and therefore will not play a major role in the integrability conditions. Similarly, when $d_P=d_Q$, the function $\Rpq$ could be up to 
some constant the importance $\md Q/\md P$. Note that bounding directly the volume of balls $B(z,2r)$ under $Q$ by the volume of 
$B(z,r)$ under $P$ is more general, given the first part of Condition \ref{cond:A2}, than bounding $Q(B(z,2r))$ by $q(z)(2r)^{d_Q}$ 
for some measurable function $q$, as we did in Condition \ref{cond:X2}, and it allows for more situations as we shall discuss in Section 
\ref{ratio}. This formulation was also used in the context of transfer exponents; see \citet{kpotufe2021marginal}, Example $4$.

Finally, Condition \ref{cond:A3} features integrability conditions that control the transferability between the distributions $P$ and $Q$. 
The constant $\Cvar$ simply controls the growths of the functions $g,\,\s^2$ and $v$. The next two integrability conditions differ only
through the functions $\s^2$, $g$, and $\mathcal G_\ell$, and they are therefore equivalent when these functions are bounded away from 
zero and infinity. Since the function $p$ is bounded above by $r_0^{-d_P}$, the last three integrability conditions in \ref{cond:A3} 
become stronger as $\gres,\,\gcov$ and $\gbias$ increase, which also improves the convergence rates; see Theorems \ref{variance} and 
\ref{bias}. This can be related to the comments on the range of the transfer function in \citet{zamolodtchikov2026minimax}.
The minimum exponents required to achieve parametric rates are $\gres=\gcov=1$ and $\gbias=1/2$. With these parameters, setting the functions
$\Delta_1(z):=g(z)^2+v(z)^{2DK^*}(\mathcal G_\ell(z)^2+\s^2(z))$ and $\Delta_2(z):=\va{g(z)}+v(z)^{DK^*}\mathcal G_\ell(z)$, the integrability
conditions require that $$\int_\X\Delta_1(z)\,\Rpq(z)\md Q(z)<\infty,\text{ and }\int_\X\Delta_2(z)\,\frac{\md Q(z)}{p(z)^{1/2}}<\infty.$$ 
They are similar to Condition \ref{cond:X3}, but they also incorporate the potential growths of the regression function $g$ and the
conditional variance $\s^2$, in the case where they are not globally bounded, or when they both approach zero.

\bigskip

Note that in the case where the two distributions have the same intrinsic dimension, that is, $d_Q=d_P$, when the integrability conditions 
are satisfied for $\gres=1$ and some $\gcov\ge 1$, and when we choose the classical pointwise matching estimator with $L=0$, for 
which $D=0$, then the multiplicative amplitude $v$ does not appear in the integrability conditions, so that the related inequality 
in Condition \ref{cond:A3} can be omitted.

\bigskip

We now analyse the rates of convergence of the estimator $\hat e_L(h)$, starting with its conditional variance.

\begin{thrm}[Conditional Variance]\label{variance}
Suppose that we have Conditions \ref{cond:A1} to \ref{cond:A3}, and that we set the parameter $k\ge(2D+1)K^*$. Then there exists a constant
$C\in\Rs$, depending only on $d$ and $L$, such that the conditional variance $V_{m,n}:=\hat e_L(h)-\E[\hat e_L(h)\mid(X_i)_{i\in\ieu}]$ of
the estimator $\hat e_L(h)$ satisfies $$\E[V_{m,n}^2]\le C\left[\frac\Cvar m+\frac{\Cres}{k^{1-\gres}\,(n+1)^{\gres}}\right],$$ 
with the constants $\Cvar,\Cres$, and the exponent $\gres$ from Condition \ref{cond:A3}.
\end{thrm}

Theorem \ref{variance} shows that the conditional variance of the estimator $\hat e_L(h)$ achieves the parametric rate of convergence 
$m^{-1}+n^{-1}$ as soon as we can take $\gres=1$ in Condition \ref{cond:A3}. It first requires $d_Q\ge d_P$, and therefore $d_Q=d_P$, since
the inequality $d_Q\le d_P$ is always implicit, as mentioned above. Second, taking $\gres=1$ requires the integrability condition
$$\int_\X v(z)^{2DK^*}\s^2(z)\,\Rpq(z)\md Q(z)<\infty,$$ which is similar to the condition $\int_\X\frac{f_Q(z)^2}{f_P(z)}\md z<\infty$ 
from Section \ref{simple}, when the multiplicative amplitude and the conditional variance $\s^2$ are globally bounded. 
We now analyse the conditional bias of the estimator $\hat e_L(h)$.

\begin{thrm}[Conditional Bias]\label{bias}
Suppose that we have Conditions \ref{cond:A1} to \ref{cond:A3}, and that we set the parameter $k\ge(2D+1)K^*$. Then there exists a constant
$C\in\Rs$, depending only on $d$ and $L$, such that the conditional bias $B_n:=\E[\hat e_L(h)\mid(X_i)_{i\in\ieu}]-e(h)$ of the estimator 
$\hat e_L(h)$ satisfies $$\E[B_n^2]\le C\left[\Ccov\left(\frac{k+1}{n+1}\right)^{\gcov}+\Cbias^2\left(\frac{k+1}{n+1}\right)^{2\gbias}
\right],$$ with the constants $\Ccov,\Cbias$, and the exponents $\gcov,\gbias$ from Condition \ref{cond:A3}.
\end{thrm}

Theorem \ref{bias} shows that taking $k=1$ can be optimal to minimise the upper bound on the rate of the conditional bias. The inequality
$\gbias\le\ell/d_P$ in Condition \ref{cond:A3} translates the effects of the regularity of the regression function and the intrinsic 
dimension of $P$ on the bias of the $k$-NN estimator. It means that the rate for the second-order conditional bias is always controlled 
by $n^{-2\ell/d_P}$, which is a standard non-parametric rate, no matter the singularity between $P$ and $Q$. However, the bias term can 
be larger than $n^{-2\ell/d_P}$ when the fourth part of Condition \ref{cond:A3} is not satisfied for $\gbias=\ell/d_P$. Setting $k=1$, 
the conditional bias of the estimator $\hat e_L(h)$ is controlled by the parametric rate of convergence $n^{-1}$ as soon as we can take 
$\gbias\ge 1/2$ and $\gcov\ge 1$. First, it requires $d_P\le 2\ell$ and $d_P\le d_Q+2\ell$, where the second inequality is automatic when 
$d_P=d_Q$. Second, it requires the two integrability conditions $$\int_\X(g(z)^2+v(z)^{2DK^*}\mathcal G_\ell(z)^2)\,\Rpq(z)\md Q(z)<\infty,
\text{ and }\int_\X(\va{g(z)}+v(z)^{DK^*}\mathcal G_\ell(z))\,\frac{\mathrm dQ(z)}{p(z)^{1/2}}<\infty.$$ 
They are similar to Condition \ref{cond:X3}, but they also incorporate the growths of $g$ and its local H\"older constant.

Combining the two previous results together, we can conclude when the estimator $\hat e_L(h)$ for the expectation $e(h)$ achieves 
the parametric rate of convergence.

\begin{coro}[Mean-squared error]\label{mse}
Suppose that we have Conditions \ref{cond:A1} to \ref{cond:A3}, with the parameters $d_Q=d_P\le 2\ell$, $\gres=1,\,\gcov\ge 1$, and
$\gbias\ge 1/2$. Suppose further that $k\ge(2D+1)K^*$, $k^{\gcov}\le\alpha n^{\gcov-1}$, and $k^{2\gbias}\le\alpha n^{2\gbias-1}$,
for some constant $\alpha\in\Rs$. Then there exists a constant $C\in\Rs$ such that the mean-squared error of the estimator $\hat e_L(h)$
satisfies $$\E[(\hat e_L(h)-e(h))^2]\le C(m^{-1}+n^{-1}).$$ 
\end{coro}

Regarding the classical matching estimator, taking a small parameter $k$ is optimal for the mean-squared error. Note that the parameter $L$ 
need not be chosen according to the regularity or the sample sizes, since to achieve parametric rates of convergence, we simply 
need $L+1\ge d_P/2$, and choosing a higher $L$ than necessary only impacts the constants in the convergence rates. Therefore, it is possible
to always set $L:=\floor{d/2}$, and the conditional bias of the estimator $\hat e_L(h)$ will automatically be negligible compared to its
conditional variance with enough regularity on the regression function and enough strong integrability conditions.

Although the theoretical results require $k\ge(2D+1)K^*$, which grows rapidly with the dimension $d$, practice shows that a
smaller parameter $k$ seems to suffice, as \citet{holzmann2026multivariate} mentioned. However, we must have $k\ge K^*$ for the local
polynomial estimator to be well-defined, otherwise the matrix $M(z)$ cannot be invertible. Because the quantity $K^*(d,\floor{d/2})$ grows
exponentially with the dimension $d$, the rates of convergence still deteriorate through constants as the dimension $d$ increases. 
If the set $\X$ is a known $d_0$-dimensional submanifold of $\R^d$, we may restrict the least-square minimisation for $\hat G_L(z)$ to the
tangent space of $\X$ at the point $z$. Doing this reduces the constants $K^*(d,L)$ and $D(d,L)$ to $K^*(d_0,L)$ and $D(d_0,L)$,
respectively, without changing the rates of convergence. Adaptive estimators that eliminate the impact of the overlying dimension $d$ on
the constants in the convergence rates of the estimators are left to future research.

\subsection{Faster rates for the matching estimator}\label{faster}

In the previous section, we showed that the locally linear estimator $\hat e_1(h)$ was well suited when the regularity parameter 
$\ell$ of the regression function was between $1$ and $2$. In this section, we prove that in this case, the same rates of convergence 
can also be achieved by the simpler matching estimator $\hat e_0(h)$, under some additional conditions. This improvement in the rates of the
conditional bias of the matching estimator was the main subject of \citet{viel2025convergence}. As previously, their results can be naturally
extended to this paper's framework, without assuming that the source distribution has a density with respect to the Lebesgue measure that is
bounded away from zero on its support. 

We assume in Condition \ref{cond:A1} that the function $g$ has first-order derivatives 
that are locally $(\ell-1)$-H\"older continuous, where $\ell\in(1,2]$. We consider two additional conditions.
\begin{enumerate}[label=(A4), wide=0.5em, leftmargin=*]
\item\label{cond:A4} There exist measurable functions $\delta:\X\to[0,r_0]$ and $\Lambda_P:\X\to\R_+$ such that for $Q$-almost all $z\in
\X$ and all $r\in[0,\delta(z)]$, setting the signed measure $P_z:=P(z+\cdot)-P(z-\cdot)$, we have $$\va{P_z}(B(0,r))\le 2\Lambda_P(z)\,
P(B(z,r))\,r^{\ell-1}.$$ Besides, the following integral is finite:
$$\Cbias':=\int_\X\norm{\nabla g(z)}\Lambda_P(z)\,p(z)^{-\gbias}\md Q(z)<\infty.$$
\end{enumerate}
\begin{enumerate}[label=(A5), wide=0.5em, leftmargin=*]
\item\label{cond:A5} The following supremum is finite: $$\Cbound:=\underset{L\in\Rs}\sup L^{\gbias-1/d_P}\int_{\delta<r_0}\norm{\nabla g(z)}
\,p(z)^{-1/d_P}\exp(-Lp(z)\delta(z)^{d_P})\md Q(z)<\infty.$$
\end{enumerate}

Condition \ref{cond:A4} translates into a local log-H\"older regularity on the source distribution $P$. Let $\X^c$ denote the complement of the 
set $\X$ in $\Rd$, and let $\delta(z,\X^c):=\inf\{\norm{z-y}\mid y\in\X^c\}$ denote the distance between the point $z\in\Rd$ and the set $\X^c$. 
If $P$ has density with respect to the Lebesgue measure on $\Rd$, then we can set $\delta(z):=\delta(z,\X^c)$ and
$$\Lambda_P(z):=\underset{y\in B(z,r_0)\cap\X}\sup\,\frac{\va{f_P(y)-f_P(z)}}{f_P(y)\norm{y-z}^{\ell-1}}.$$ 
Meanwhile, Condition \ref{cond:A5} is another integrability condition that incorporates boundary effects. These two conditions are also discussed
in Section \ref{ratio}.
These two conditions extend the framework of \citet{viel2025convergence}, where the regression function $g$ was supposed to have continuous
second-order derivatives on a compact set, which implies that its gradient was Lipschitz continuous, corresponding to the case where $\ell=2$ 
here. In addition, the source density function $f_P$ was supposed to be Lipschitz continuous on its support, and because it was also supposed to 
be bounded away from zero, it was automatically log-Lipschitz continuous. Finally, when the function $\norm{\nabla g}$ is bounded and $p$ is
bounded away from zero, the finiteness of the quantity $\Cbound$ in Condition \ref{cond:A5}, with the maximal exponent $\gbias=2/d_P$, implies 
the geometric condition $(A)$ in \citet{viel2025convergence}.

\begin{thrm}[Conditional Bias]\label{better_bias}
Suppose that we have Conditions \ref{cond:A1} to \ref{cond:A5} with the parameter $\ell\in(1,2]$, and with the exponents $\gcov\in[0,(d_Q+2)/d_P]$ 
and $\gbias\in[1/d_P,\ell/d_P]$. Then there exists a universal constant $C\in\Rs$ such that the conditional bias $B_n:=\E[\hat e_0(h)\mid
(X_i)_{i\in\ieu}]-e(h)$ of the estimator $\hat e_0(h)$ satisfies
$$\E[B_n^2]\le\Ccov\left(\frac{k+1}{n+1}\right)^{\gcov}+(\Cbias^2+\Cbias^{'2}+\Cbound^2)\left(\frac{k+1}{n+1}\right)^{2\gbias},$$ 
with the constants $\Ccov$ and $\Cbias$ from Condition \ref{cond:A3}, $\Cbias'$ from Condition \ref{cond:A4}, and $\Cbound$ from Condition
\ref{cond:A5}.
\end{thrm}

Theorem \ref{better_bias} recovers similar rates of convergence for the conditional bias of the matching estimator than in Theorem \ref{bias}, 
if we consider the way in which they depend on the transferability exponent $\gbias$. The difference lies in the fact that we now 
allow $1/d_P<\gbias\le 2/d_P$ instead of $\gbias\le 1/d_P$, as in Theorem \ref{bias}. This is particularly relevant when the intrinsic 
dimension $d_P$ is equal to $3$ or $4$, if we want to reach the critical exponent $\gbias=1/2$.

\begin{coro}[Mean-squared error]\label{better_mse}
Suppose that we have Conditions \ref{cond:A1} to \ref{cond:A5} with $\ell=2$, $d_Q=d_P\le 4$, $\gres=1,\,\gcov\ge 1$, and $\gbias\ge 1/2$.
Suppose further that $k^{\gcov}\le\alpha n^{\gcov-1}$ and $k^{2\gbias}\le\alpha n^{2\gbias-1}$ for some constant $\alpha\in\Rs$. Then there 
exists a constant $C\in\Rs$ such that the mean-squared error of the estimator $\hat e_0(h)$ satisfies 
$$\E[(\hat e_0(h)-e(h))^2]\le C(m^{-1}+n^{-1}).$$ 
\end{coro}

We can deduce from this theorem that the final matching estimator $\hat e_0(h)$ can achieve parametric rates of convergence when
the intrinsic dimensions are equal and no more than $4$, and under transferability conditions similar to those of 
section \ref{polynomial}.

\section{Application to the estimation of average treatment effects}\label{ate}

We now discuss how the results from the previous sections can be applied to the estimation of average treatment effects. We consider a 
random vector $(X,W,Y(1),Y(0))$ taking values in $\R^d\times\{0;1\}\times\R\times\R$. In this random vector, $X$ is a vector of 
covariates for an individual, $W$ corresponds to whether the individual received an active ($W=1$) or a passive treatment ($W=0$), and 
$Y(1)$ and $Y(0)$ are the two potential outcomes for the active and passive treatment, respectively. One objective is then to 
estimate the average treatment effect $\mu:=\E[Y(1)-Y(0)]$, when for each individual $i\in\ieu[N]$, only the quantities 
$(X_i,W_i,Y_i(W_i))$ are observed, but not the potential outcome $Y_i(1-W_i)$ for the treatment the individual did not receive. The 
missing outputs $Y_i(1-W_i)$ are then replaced by an average of the observed outputs $Y_j(1-W_i)$ for the $k$ units $j$ that have the 
covariates closest to $X_i$ between the individuals given the treatment $W_j=1-W_i$, as proposed by \citet{abadie2006large}. The 
average treatment effect context then relates to our problem, the source and target distributions being the conditional 
distributions of $X$ given $W=1$ and $X$ given $W=0$, or vice versa. For the covariate shift setup, we assume that the treatment 
$W$ is independent from the outcomes $(Y(1),Y(0))$ given the covariates $X$, a situation called \emph{unconfoundedness}.

In this scenario, it is commonly assumed that the covariate vector $X$ lives in a bounded domain $\X$ and has a density with respect to 
the Lebesgue measure that is bounded away from zero; see, for example, \citet{abadie2006large} and \citet{holzmann2026multivariate}. However, we 
can show that the same rates of convergence for the same estimators hold under milder assumptions on the covariate support and 
distribution. Although our main focus is the average treatment effect $\mu$, the same results are valid for the average treatment effect 
on the treated $\mu^{\textup{treat}}:=\E[Y(1)-Y(0)\mid W=1]$. 

We assume that we have i.i.d. replications $(X_i,W_i,Y_i(1),Y_i(0))_{i\in\ieu[N]}$ of the vector $(X,W,Y(1),Y(0))$. Write $S$ for the marginal 
distribution of $X$, supported on some possibly unknown Borel set $\X\subset\R^d$. Fix $k\in\ieu[N]$, for $S$-almost all $x\in\X$ and
all $w\in\{0;1\}$, set $\hat\tau_{k+1,w}(x)$ the distance between the point $x$ and its $(k+1)$th-nearest neighbour among the observations 
$X_i$ such that $W_i=w$, with the convention $\hat\tau_{k+1,w}(x):=\infty$ if there are not enough observations. Keeping the notation from
Section \ref{polynomial}, in particular the set of multi-indices $N_L$ and the vector of monomials $\zeta(z,T)$, we define the final
estimator $$\hat\mu_L:=\frac 1N\sum_{i=1}^N(2W_i-1)(Y_i-\hat g_{L,1-W_i}(X_i)),\text{ where}$$
$$\hat g_{L,w}(x):=\1{\hat\tau_{k+1,w}(x)\le r_0}e_0^\top\underset{b\in\R^{N_L}}\argmin\sum_{\underset{W_i=w}{i=1}}^N(Y_i-b^\top
\zeta(x,X_i))^2\1{\norm{X_i-x}<\hat\tau_{k+1,w}(x)},$$ for $S$-almost all $x\in\X$ and all $w\in\{0;1\}$.
Apart from censoring when $\hat\tau_{k+1,W_i}(X_j)>r_0$, which is not necessary in most cases and could be omitted in practice, the 
estimator $\hat\mu_L$ coincides with the local polynomial estimator from \citet{holzmann2026multivariate}. In particular, when $L=0$, we 
recover the pointwise matching estimator, introduced in \citet{abadie2006large},
$$\hat g_{0,w}(x)=\frac 1k\sum_{\underset{W_i=w}{i=1}}^NY_i\;\1{\norm{X_i-x}<\hat\tau_{k+1,w}(x)}.$$

We set, for $S$-almost all $x\in\X$ and all $w\in\{0;1\}$, the regression function $g_w(x):=\E[Y(w)\mid X=x]$. 
We consider the following conditions. 
\begin{enumerate}[label=(A\arabic*'), wide=0.5em, leftmargin=*]
\item\label{cond:A1'} There exist $\ell\in(0,L+1]$ and measurable functions $\mathcal G_{\ell,0},\mathcal G_{\ell,1},\s_0^2,
\s_1^2:\X\to\R_+$ such that for $S$-almost all $x\in\X$ and all $w\in\{0;1\}$, the conditional variance $z\mapsto\Var(Y(w)\mid X=z)$ 
is bounded by $\s_w^2(x)$ on $B(x,r_0)\cap\X$, and the regression function $g_w$ is the restriction of a function having 
derivatives of order $\ceil\ell-1$ that are all $(\ell+1-\ceil\ell)$-H\"older continuous on $B(x,3r_0)\cap\Conv(\X)$ with constant 
$\mathcal G_{\ell,w}(x)\ceil{\ell-1}!$.
\item\label{cond:A2'} There exist $d_0\in\Rs$ and measurable functions $p:\X\to\Rs,v:\X\times[0,r_0]\to[1,\infty[$ such that for $S$-almost all
$x\in\X$, all $r\in[0,r_0]$, and all Borel sets $\mathcal A\subset B(0,3)^{K^*}$, we have \begin{align*} 
S(B(x,r)) & \ge p(x)r^{d_0},\text{ and} \\ \Pb\left(\left(\frac{X_1-x}r,\dots,\frac{X_{K^*}-x}r\right)\in\mathcal A\right) &
\le(v(x,r)p(x)r^{d_0})^{K^*}\va{\mathcal A},
\end{align*}
where $\va{\mathcal A}$ designates the Lebesgue measure of the Borel set $\mathcal A\subset\R^{dK^*}$. \\
Besides, there exist measurable functions $\eta_0,\eta_1:\X\to\,(0,1)$ such that for $S$-almost all $x,z\in\X$ such that $\norm{z-x}
\le 2r_0$, and all $w\in\{0;1\}$, $$\eta_w(x)\le\Pb(W=w\mid X=z)\le 1-\eta_{1-w}(x).$$
\item\label{cond:A3'} There exist $\gres\in[0,1],\,\gcov\in[0,1+2\ell/d_0]$, and $\gbias\in[0,\ell/d_0]$, such that for all $w\in\{0;1\}$,
the following integrals are finite:
$$\int_\X(\norm{g_w}_{\infty,B(x,r_0)\cap\X}^2+\sigma_w^2(x))v(x)^{2DK^*}\md S(x)<\infty,$$
$$\int_\X\left[\sigma_w^2(x)(1-\eta_w(x))^2v(x)^{2DK^*+1}\right]\frac{p(x)^{1-\gres}}{\eta_w(x)^{\gres}}\md S(x)<\infty,$$
$$\int_\X\left[(g_w(x)^2+\mathcal G_{\ell,w}(x)^2v(x)^{2DK^*})(1-\eta_w(x))^2v(x)\right]\,\frac{p(x)^{1-\gcov}}{\eta_w(x)^{\gcov}}
\md S(x)<\infty,\text{ and}$$ 
$$\int_\X\left[(\va{g_w(x)}+\mathcal G_{\ell,w}(x)v(x)^{DK^*})(1-\eta_w(x))\right]\,\frac 1{(\eta_w(x)p(x))^{\gbias}}\md S(x)<\infty,$$
where we have set, for $S$-almost all $x\in\X$, $v(x):=\underset{r\in[0,r_0]}\sup\frac{p(x)r^d}{S(B(x,r))}v(x,r)\le\underset{r\in[0,r_0]}\sup
v(x,r)$.
\end{enumerate}

\bigskip

This list of conditions is similar to that in Section \ref{general}, with three major differences. \begin{itemize}
\item We have two regression functions and conditional variances due to the presence of the two potential outcomes $Y(1)$ and $Y(0)$.
\item We only have one unconditional distribution $S$ for the covariate vector $X$. As a result, the source and target distributions, 
which are the conditional distributions of $X$ given $W=1$ and of $X$ given $W=0$, or vice versa, naturally share the same intrinsic 
dimension $d_0$, which is the intrinsic dimension of $S$ as given in Condition \ref{cond:A2'}.
\item The transferability between the source and target probability measures is controlled via the functions $\eta_0$ and $\eta_1$, which give
local bounds to the conditional probability to assign a treatment $w$ given the covariate vector $X$, translating the selection bias. 
Then the last three parts of Condition \ref{cond:A3'} measure the integrability of the function $x\mapsto 1/(\eta_w(x)p(x))$ with respect
to $\mathrm dS(x)$ or $p(x)\mathrm dS(x)$.
\end{itemize}

\bigskip

In the literature, the following global assumptions are common. \begin{itemize}
\item The conditional variances $\Var(Y(0)\mid X)$ and $\Var(Y(1)\mid X)$, the regression functions $g_0$ and $g_1$, and all their
derivatives, are bounded on $\X$.
\item The distribution $S$ has a density with respect to the Lebesgue measure of $\R^d$ that is bounded away from zero and infinity on the
compact set $\X$.
\item The conditional probability of assigning a given treatment $\Pb(W=w\mid X)$ is bounded away from zero and one on $\X$.
\end{itemize}
In this case, the integrability conditions in \ref{cond:A3'} are automatically satisfied. Our list of conditions is therefore more 
general and enables the study of treatment effect estimation when the covariate vector is supported on an unbounded domain, on a known or
unknown submanifold of $\R^d$, and when the conditional probability to assign a given treatment may approach zero or one.

We are now in a position to state the following result.
\begin{thrm}\label{ate_poly}
Suppose that we have Conditions \ref{cond:A1'} to \ref{cond:A3'}, and that we have set the parameter $k\ge(2D+1)K^*$. Then there exists a 
constant $C\in\Rs$ such that the mean-squared error of the estimator $\hat\mu_L$ satisfies 
$$\E[(\hat\mu_L-\mu)^2]\le C\left(\frac 1N+\frac 1{k^{1-\gres}N^{\gres}}+\left(\frac kN\right)^{\gcov}+\left(\frac kN\right)^{2\gbias}
\right).$$
\end{thrm}

In particular, if we have $L+1\ge d_0/2$, which can be ensured by setting $L\ge\ceil{d/2}-1$, and if the number $k$ of nearest neighbours
grows slowly, if at all, with the sample size $N$, then the estimator $\hat\mu_L$ for the average treatment effect $\mu$ can achieve the
parametric rate of convergence $N^{-1}$ in all dimensions. Achieving this rate requires two main elements. First, the regression functions
$g_0$ and $g_1$ must be sufficiently regular. If $\ell\in\Rs$ denotes the minimum between the H\"older regularity parameters of $g_0$ and
$g_1$, then we ask $d_0\le 2\ell$, where $d_0$ is the intrinsic dimension of the covariate distribution $S$. Second, we must control
the selection bias, which is measured by the integrability of the functions $1/\eta_0$ and $1/\eta_1$. More precisely, Condition
\ref{cond:A3'} must be satisfied for $\gres=1,\,\gcov\ge 1$, and $\gbias\ge 1/2$. 
In the case where, for $w\in\{0;1\}$, the functions $\s^2_w,g_w,\mathcal G_{\ell,w}$, and $v$ are bounded away from zero and infinity on
$\X$, then the integrability conditions become 
$$\int_\X\frac 1{\eta_w(x)}\md S(x)<\infty,\text{ and }\int_\X\frac 1{(\eta_w(x)p(x))^{1/2}}\md S(x)<\infty.$$
If, in addition, the distribution $S$ has density with respect to the Lebesgue measure on $\R^d$, with density equal to $p$ up to a 
constant, then the integrability conditions become $$\int_\X\frac{p(x)}{\eta_w(x)}\md x<\infty,\text{ and }\int_\X\left(\frac{p(x)}{
\eta_w(x)}\right)^{1/2}\md x<\infty.$$ As a result, it implies that the conditional selection probability function $x\mapsto\Pb(W=w\mid
X=x)$ is allowed to approach zero, but not faster than the covariate probability density $p$.

\section{Conditions on the relative behaviour of the source and target distributions}\label{ratio}

\subsection{Usual distributions on the Euclidean space}\label{full_dim}

In this section, we study the conditions \ref{cond:A2} to \ref{cond:A5} for distributions $P$ and $Q$ that have density with respect to the
Lebesgue measure on $\R^d$, denoted respectively by $f_P$ and $f_Q$. To show that these conditions are the most general, we first give intuitive 
interpretations before investigating what they become for common distributions on $\R^d$. These conditions can be put into two categories: 
the volume bounding conditions and the integrability conditions. We start by discussing volume bounding. 

In the first part of Condition \ref{cond:A2}, writing $P(B(z,r))\ge p(z)r^{d_P}$, here $d=d_P$ is the dimension of the overlying space 
$\R^d$, and $p$ can be up to some constant the minimum of the density $f_P$ over a subset of $B(z,r_0)$ that always intersect $B(z,r),
r\in[0,r_0]$ with a volume that is comparable to that of $B(z,r)$. Boundary effects are discussed at the end of this section. When
there are no boundary effects, $p$ can often be set to the density $f_P$ multiplied by a positive constant. The last part of Condition
\ref{cond:A2}, writing $Q(B(z,2r))\le\Rpq(z)P(B(z,r))$ when $d_Q=d_P$, works well with the previous lower bound. Indeed, it is then 
sufficient to show that there exists a function $q:\X\to\R_+$ such that for $Q$-almost all $z\in\X$ and all $r\in[0,2r_0]$, $Q(B(z,r))\le
q(z)r^d$, and then set $\Rpq(z):=2^dq(z)/p(z)$. For example, $q$ can be the maximum of $f_Q$ over $B(z,2r_0)$ multiplied by the volume 
of the unit ball. However, the general formulation works better with boundary effects; see the proof of Proposition \ref{uniform}.
Similarly, the last part of Condition \ref{cond:A2} is satisfied if, for instance, $v(z,r)p(z)$ is the maximum of $f_P$ over $B(z,r)$. 
Although it is classical in the literature to assume that $v$ is a bounded function \citet{zamolodtchikov2026minimax}, our results are
still valid in the case where any power of $v$ is negligible compared to $f_Q$; see the proof of Proposition \ref{gaussian}.

Finally, in Condition \ref{cond:A4}, writing $\va{P_z}(B(0,r))\le 2\Lambda_P(z)P(B(z,r))\,r^{\ell-1}$ for all $r\in[0,\delta(z)]$, 
where $P_z:=P(z+\cdot)-P(z-\cdot)$, since $P$ has density $f_P$ with respect to the Lebesgue measure on $\R^d$, we can write
$$\va{P_z}(B(0,r))=\int_{B(0,r)}\va{f_P(z+x)-f_P(z-x)}\md x.$$
If the density $f_P$ is H\"older-continuous on $B(z,r)$, with exponent $\ell-1\in(0,1]$ and constant $L(z)$, then $$\va{P_z}(B(0,r))
\le 2r^{\ell-1}L(z)\va{B(0,r)}\le 2r^{\ell-1}\va{B(0,1)}\frac{L(z)}{p(z)}P(B(z,r)),$$
using Condition \ref{cond:A2}. Condition \ref{cond:A4} is then satisfied for $\Lambda_P(z):=\va{B(0,1)}L(z)/p(z)$, which, with the
intuition behind Condition \ref{cond:A2}, can be, up to a negligible function, the log-Lipschitz norm $\norm{\nabla f_P(z)}/f_P(z)$ of a
differentiable density $f_P$, when $\ell=2$. This reasoning works if the density $f_P$ is H\"older-continuous on $B(z,\delta(z))$, and we
can take into account boundary effects by setting $\delta(z)$ the minimum between $r_0$ and the distance $\delta(z,\X^c)$ between $z$ and
the complement of $\X$.

\bigskip

We now discuss the integrability conditions. Although they are easier to state when the functions $\s^2,\,g$ and all its derivatives are
bounded, they do not change if these functions have growths that are negligible compared to the decay of $Q$, for instance polynomial
growths against a sub-exponential tail for $Q$. The same applies to the functions $v$ and $\Lambda_P$ from Conditions \ref{cond:A2} and 
\ref{cond:A4}, which are often negligible compared to the decay of $Q$. Following the intuition behind the functions $p$ and $\Rpq$, the
integrability conditions are often equivalent to conditions of the form $$\int_\X\frac{f_Q(z)^\alpha}{f_P(z)^\gamma}\md z<\infty,$$ 
with $\alpha\in\{1,2\}$ and $\gamma\in\Rs$. With a finite parameter $k$, we saw that the conditional bias of the final estimator
$\hat e_L(h),\,L\in\N$ was comparable to the conditional variance when the integrability conditions were satisfied for the couples of
exponents $(\alpha,\gamma)=(1,1/2)$ and $(\alpha,\gamma)=(2,1)$. In both cases, the exponent on $f_Q$ is twice the exponent on $f_P$, 
which, as stated above, allows $Q$ to have a heavier tail than $P$ as long as it is not twice as heavy.

The different integrability condition is Condition \ref{cond:A5}. It features an integral over the extended boundary
$\{z\in\X\mid\delta(z)<r_0\}$, that is, points whose distance to the boundary of $\X$ is less than $r_0$. This condition is then
trivially satisfied when $\X$ has no boundary, in other words when the density $f_P$ is regular on the whole space $\X=\R^d$. It is 
also not present in section \ref{polynomial} since the local polynomial fitting takes into account boundary effects. See the 
end of this section for a further discussion of this condition. 

\bigskip

We now go through a series of examples to see what all these conditions become for common distributions on $\R^d$, starting with
multivariate Gaussian distributions.

\begin{prop}[Gaussian distributions]\label{gaussian}
Let $P$ and $Q$ follow multivariate Gaussian distributions $P\sim\mathcal N_d(\mu_P,\Sigma_P)$ and $Q\sim\mathcal N_d(\mu_Q,\Sigma_Q)$. 
Assume that the norms of $\s^2,\,g$ and all its derivatives are controlled above by some function $z\mapsto C\exp(A\norm z)$ and below by
$z\mapsto c\exp(-a\norm z)$, where $a,c,A,C\in\Rs$. Then the couple $(P,Q)$ satisfies the conditions \ref{cond:A2} to \ref{cond:A5}, for
$\gres,\gcov,\gbias\in\R_+$ if and only if $2\Sigma_Q^{-1}>\max\{\gres;\gcov\}\Sigma_P^{-1}$ and $\Sigma_Q^{-1}>\gbias\Sigma_P^{-1}$, in the
sense of positive semi-definite matrix inequalities. In particular, with the critical choices $\gcov=\gres=1$ and $\gbias=1/2$, the couple
$(P,Q)$ satisfies the lists of conditions if and only if $2\Sigma_Q^{-1}>\Sigma_P^{-1}$.
\end{prop}

\paragraph{Note} Gaussian vectors give an example of multivariate distributions with unbounded support and no boundary effects. They are
often excluded in the context of regression on unbounded domains because they do not satisfy the local mass property
\citet{zamolodtchikov2026minimax}. With our notation, it is translated by the fact that the multiplicative amplitude $v(z)$ is not bounded 
as $z$ approaches infinity. However, the critical transferability condition $2\Sigma_Q^{-1}>\Sigma_P^{-1}$ is a good illustration of the 
paradigm "the target distribution is allowed to have a heavier tail than the source distribution, as long as it is not twice as heavy" 
as stated in Section \ref{simple}.

\bigskip

Next, we consider the example of vectors with independent coordinates, with one coordinate following a Gamma distribution and the others 
following uniform distributions. This example enables the study of univariate Gamma distributions while not restricting to the case $d=1$.

\begin{prop}[Gamma distributions]\label{gamma}
Let $P\sim\Gamma_1(\mu_P,s_P)$ and $Q\sim\Gamma_1(\mu_Q,s_Q)$, where the distribution $\Gamma_1(\mu,s),\mu\in\Rs,s\in[1,\infty[$ is defined
by the density $z\mapsto\frac{\mu^s}{\Gamma(s)}z_1^{s-1}e^{-\mu z_1}\1{z_1\in\Rs,z_i\in[0,1],i\in\ieu[d][2]}$, representing univariate 
Gamma distributions in an overlying space of dimension $d$. Assume that $\s^2,\,g$ and all its derivatives have polynomial growths on
$\R_+\times[0,1]^{d-1}$. Then the couple $(P,Q)$ satisfies the conditions \ref{cond:A2} and \ref{cond:A3}, for $\gres,\gcov,\gbias\in\R_+$ 
if and only if $2\mu_Q>\max\{\gres;\gcov\}\mu_P$, $\mu_Q>\gbias\mu_P$, $2s_Q>\max\{\gres;\gcov\}s_P$, and $s_Q>\gbias s_P$. In particular, with
the critical choices $\gcov=\gres=1$ and $\gbias=1/2$, the couple $(P,Q)$ satisfies the the conditions \ref{cond:A2} and \ref{cond:A3} if and
only if $2\mu_Q>\mu_P$ and $2s_Q>s_P$.
\end{prop}

\paragraph{Note} This example features a mix of several regimes. First, we have the Gamma tail regime $z_1\to\infty$ which gives the
intuitive critical transferability condition $2\mu_Q>\mu_P$. Second, we have the power regime $z_1\to 0$ which gives the similar 
transferability condition $2s_Q>s_P$. It shows that we can also handle densities that go to zero near the boundary of the support. 
Finally, we also have boundary effects due to the uniform coordinates $z_i,i\in\ieu[d][2]$ and the regime $z_1\to 0$. However, due to 
the simple geometry of the support, the boundary effects are easily dealt with and do not add any additional conditions. 
See Proposition \ref{uniform} for a discussion of more complex boundary effects.

\paragraph{Note} If we want to apply the results from Section \ref{faster} and therefore need Conditions \ref{cond:A4} and \ref{cond:A5}, 
then we need the slightly stronger condition $2s_Q>s_P+1$. This is due to the fact that power densities of the form $z_1^s$ lack log-regularity
when $z_1\to 0$, which is needed in Section \ref{faster}.

\bigskip

We now address the case where the regression function has a growth towards infinity that is comparable to the decay of the distributions.

\begin{prop}[Pareto distributions with a polynomial regression function]\label{pareto}
Let $P$ and $Q$ follow Pareto distributions in overlying space of dimension $d,\,P\sim\operatorname{Par}_1(\mu_P)$ and $Q\sim
\operatorname{Par}_1(\mu_Q)$, where the distribution $\operatorname{Par}_1(\mu),\mu\in\Rs$ is defined by the density $z\mapsto z_1^{-\mu
-1}/\mu\,\1{z_1\in[1,\infty),z_i\in[0,1],i\in\ieu[d][2]}$. Assume that $\s^2$ is bounded away from zero and infinity and that for all
$z\in[1,\infty)\times[0,1]^{d-1}$ and all $s\in\ieu[\ceil\ell]$, $cz_1^M\le\va{g(z)}\le Cz_1^M$ and $\norm{D^sg(z)}\le Cz_1^M$, where
$M\in\R_+$ and $c,C\in\Rs$. Then the couple $(P,Q)$ satisfies the conditions \ref{cond:A2} to \ref{cond:A5}, for $\gres,\gcov,\gbias\in\R_+$
if and only if $\mu_Q>2M$, $2\mu_Q+1>\gcov(\mu_P+1)$, $2\mu_Q+1>2M+\gres(\mu_P+1)$, and $\mu_Q>M+\gbias(\mu_P+1)$. In particular, with the 
critical choices $\gcov=\gres=1$ and $\gbias=1/2$, the couple $(P,Q)$ satisfies the lists of conditions if and only if $\mu_Q>2M$ and
$2\mu_Q>2M+\mu_P+1$.
\end{prop}

\paragraph{Note} The previous example shows that the growth of the regression function can be incorporated in the transferability conditions.
\paragraph{Note} In the particular case where the function $g$ is exactly a polynomial in $z_1$ of degree $M>d/2$, we believe that a closer 
look to the estimator without censoring when $\tak>r_0$ could lead to the critical transferability conditions $\mu_Q>2M$ and $2\mu_Q+d>
2M+\mu_P+1$, provided that we have the additional assumption $\mu_P>d/2$. It would allow the target distribution to have a heavier tail than 
the source distribution even when $\mu_Q\approx 2M$. We do not investigate this case further here.

The examples above show that our integrability conditions can be easily interpreted for common distributions. Besides, these conditions 
are local and allow the mixing of several regimes for the densities (decaying towards infinity, non-continuity, going to zero near a
finite point, etc.).

\bigskip

We now discuss the effects of the geometry of the support $\X$ on the distribution conditions, considering the region near the boundary 
of $\X$. Because the conditions are local, to study boundary effects, we may assume without loss of generality that $\X$ is compact. 
We also assume that the functions $\s^2,\,g$, and all its derivatives, are bounded away from zero and infinity.

The two conditions that suffer the most from boundary effects are the first part of Condition \ref{cond:A2}, and Condition \ref{cond:A5}, 
here writing, respectively, $$P(B(z,r))\ge p(z)r^d,\text{ and }\underset{L\in\Rs}\sup L^{\gbias-1/d}\int_{B(\partial\X,r_0)}p(z)^{-1/d}
\exp(-Lp(z)\delta(z)^d)\md Q(z)<\infty.$$ 
Condition \ref{cond:A2} is satisfied when $P$ has a density with respect to the Lebesgue measure on $\X$ that is bounded away from zero,
and when there exists a constant $c\in(0,1]$ such that for almost all $z\in\X$ and all $r\in[0,r_0]$, $$\va{B(z,r)\cap\X}\ge c\va{B(z,r)},$$
leading to the \emph{strong density assumption} \citep{audibert2007fast}. The conjunction of this assumption with the condition
$\underset{L\in\Rs}\sup L^{1/d}\int_\X\exp(-L\delta(z)^d)\md z<\infty$ was studied in \citet{viel2025convergence}, who showed it holds on
general sufficiently regular subsets $\X$ of $\R^d$. We can also consider densities that are not bounded away from zero near the boundary,
as we saw in Proposition \ref{gamma} with Gamma distributions. Finally, it is possible to relax the condition $\va{B(z,r)\cap\X}\ge c\va{
B(z,r)}$, as well as the upper bound conditions on the densities, as can be seen in the following example.

\begin{prop}\label{uniform}
Consider the set $\X:=\{z\in[0,1]^d\mid\forall i\in\ieu[d][2],z_i\le z_1^s\}$, where $s\ge 1$. Let $\mu_P,\mu_Q>-s(d-1)$ and let $P,Q$
follow the distributions with respective densities $f_P:z\mapsto(s(d-1)+\mu_P+1)z_1^{\mu_P}$ and $f_Q:z\mapsto(s(d-1)+\mu_Q+1)z_1^{\mu_Q}$ 
with respect to the Lebesgue measure on $\X$. Suppose that the functions $\s^2$ and $g$ are bounded away from zero and infinity on $\X$,
as well as $\mathcal G_\ell$ on $\Conv(\X)$. Let $\gamma\in\R_+$, then the couple $(P,Q)$ satisfies Conditions \ref{cond:A2} to \ref{cond:A5}, 
for $\gres=\gcov=\gamma,\gbias=\gamma/2$ if and only if $2\mu_Q+((2-\gamma)s-1+\gamma)(d-1)>\gamma\mu_P-1$. 
\end{prop}

With the critical choice $\gamma=1$, the list of conditions is satisfied if and only if $2\mu_Q>\mu_P-s(d-1)-1$, while the strong density
assumption $\exists c\in(0,1],\forall z\in\X,\forall r\in[0,1],\va{B(z,r)\cap\X}\ge c\va{B(z,r)}$ is not satisfied when $s>1$, and the 
target density is not bounded when $\mu_Q<0$.

This last example is particular since the function $p$ in Condition \ref{cond:A2} will not be the density function $f_P$, up to a positive
constant, as in the previous examples, but it also depends on the geometry of the support through the exponent $s$. See the proof of Proposition
\ref{uniform} for the details.

\subsection{Distributions on a manifold support}\label{manifold}

In this section, we study the conditions \ref{cond:A2} to \ref{cond:A5} when the covariate distributions $P$ and $Q$ on the Euclidean
space $\R^d$ have intrinsic dimensions smaller than $d$. This can happen, for example, when the support $\X$ of the distributions is a 
submanifold of $\R^d$ of dimension $d_0$. Note that the estimators we consider are adaptive in the sense that no knowledge of the 
support $\X$ or its dimension $d_0$ is required for the computations. 

We first consider the example of the Euclidean unit sphere $\Sp:=\{z\in\R^d\mid\sum_{i=1}^dz_i^2=1\}$. It is a smooth compact
$(d-1)$-dimensional submanifold of $\R^d$, with no boundary. Assume that the distributions $P$ and $Q$ are absolutely continuous with
respect to the Haar measure on $\Sp$, with densities that are bounded away from infinity, and a density for $P$ that is bounded away from
zero. Then there exist constants $c,C,C_{P,Q}\in\Rs$ such that for all $z\in\Sp$ and all $r\in[0,r_0]$, $$cr^{d-1}\le P(B(z,r))\le C
r^{d-1},\text{ and}\quad Q(B(z,2r))\le C_{P,Q}P(B(z,r)),$$ ensuring the validity of Condition \ref{cond:A2} with the intrinsic dimensions
$d_Q=d_P=d-1$. If the density of $P$ is Lipschitz continuous on the sphere, then Condition \ref{cond:A4} is satisfied, setting 
$\delta(z):=r_0$. Finally, Condition \ref{cond:A5} is trivially satisfied, since the submanifold $\Sp$ has no boundary, and Condition
\ref{cond:A3} is satisfied, as the source density is bounded away from zero on the compact submanifold $\Sp$.

\bigskip

We can also consider non smooth manifolds, or manifolds with boundaries. Assume that the distributions $P$ and $Q$ have density with
respect to the natural surface measure on the unit sphere for the infinity norm, defined as $\X_\infty:=\{z\in\R^d\mid\max\{\va{z_i},
i\in\ieu[d]\}=1\}$. Assume further that the two densities are bounded away from infinity and that the density of $P$ is bounded
away from zero and Lipschitz continuous on $\X_\infty$. Then the intrinsic dimensions of $P$ and $Q$ are $d_Q=d_P=d-1$. The only major
difference compared to the previous situation on the Euclidean unit sphere $\Sp$ is the presence of a boundary. Here, Condition 
\ref{cond:A4} is satisfied when $\delta(z):=\delta(z,\partial\X_\infty)$ is the distance between the point $z$ and the boundary
$\partial\X_\infty:=\{z\in\X_\infty\mid\exists i\neq j,\,\va{z_i}=\va{z_j}=1\}$. However, due to the simple geometry of the boundary,
we can show that Condition \ref{cond:A5} is satisfied.

\bigskip

Next, we consider the case of non compact manifolds. We use the description of submanifolds of $\R^d$ through graphs. Let $\X$ be the
global submanifold $\X:=\{(a,\varphi(z))\mid a\in\R^{d_0}\}$, where $\varphi:\R^{d_0}\to\R^{d-d_0}$ is a function with continuous
first-order derivatives on $\R^{d_0}$. Then the regression function $g$ on $\X$ depends only on the function $g_0:\R^{d_0}\to\R$,
defined as $g_0(a):=g(a,\varphi(a))$. Similarly, we replace the study of the distributions $P$ and $Q$, supported on the submanifold
$\X$, by that of the distributions $P_0$ and $Q_0$, supported on $\R^{d_0}$, where for a Borel set $A\subset\R^{d_0}$, we have set
$P_0(A):=P(\{(a,\varphi(a))\mid a\in A\})$, and similarly for $Q_0$. We can then write, for almost all $a\in\R^{d_0}$ and all $r\in
[0,r_0]$, setting $L(a):=\norm{\nabla\varphi}_{\infty,B_{d_0}(a,r_0)}\in\R_+$ and $z:=(a,\varphi(a))\in\X$, $$P_0\left(B_{d_0}\left(a,
\frac r{1+L(a)}\right)\right)\le P(B(z,r))\le P_0(B_{d_0}(a,r)),$$ where the norm on $\R^d$ is assumed to be any Minkowski norm.
Assuming that $P_0$ and $Q_0$ have density with respect to the Lebesgue measure on $\R^{d_0}$, we can obtain bounds of the form
$p_0(a)r^{d_0}\le P_0(B_{d_0}(a,r))$, as in Section \ref{full_dim}, which implies that the distributions $P$ and $Q$ have intrinsic
dimensions $d_Q=d_P=d_0$. In addition, the integrability conditions can be fully expressed on $\R^{d_0}$, using the function $g_0$ and the
distributions $P_0$ and $Q_0$. For example, the last part of Condition \ref{cond:A3} becomes $$\int_{\R^{d_0}}(\va{g_0(a)}+\norm{D^\ell
g_0(a)})\left(\frac{1+L(a)^{d_0}}{p_0(a)}\right)^{\gbias}\md Q_0(a)<\infty.$$

\bigskip

In general, because the volume bounding and the integrability conditions are local, they are preserved by local maps that transport 
the source and target probability measures. For a general submanifold $\X$ of $\R^d$, we fix $z\in\X$ and we assume that there exist an
orthogonal transform $\Omega_z$ of $\R^d$, an open subset $A_z$ of $\R^{d_0}$, and an $L_z$-Lipschitz function $\varphi_z:A_z\to\R^{d-
d_0},\,L_z\in\R_+$, such that $\X\cap B(z,2r_0)=\{\Omega_z(a,\varphi_z(a))\mid a\in A_z\}$. We can then replace the study of the
regression function $g$ and the distributions $P$ and $Q$, locally around $z$, by that of the function $g_0:A_z\to\R$, defined by
$g_0(a)=g(\Omega_z(a,\varphi_z(a)))$, and the measures $P_0$ and $Q_0$ on $A_z$, defined by $P_0(C):=P(\{\Omega_z(a,\varphi_z(a))\mid a
\in C\})$, and similarly for $Q_0$. For example, setting $z=\Omega_z(a_z,\varphi_z(a_z))$, if the measure $P_0$ satisfies, for all
$r\in[0,r_0],\,p_0(a_z)r^{d_0}\le P_0(B_{d_0}(a_z,r))\le p_0(a_z)v_0(a_z)r^{d_0}$, which can be shown as in Section \ref{full_dim},
then we have, for all $r\in[0,r_0]$, $$\frac{p_0(a_z)}{(1+L_z)^{d_0}}r^{d_0}\le P(B(z,r))\le p_0(a_z)v_0(a_z)\,r^{d_0}.$$

\bigskip

The above derivations show that when the covariate vectors $X$ and $X^*$ have support on a $d_0$-dimensional submanifold $\X_0$ of $\R^d$, 
we can take $d_Q=d_P=d_0$ in Condition \ref{cond:A2}, reducing the conditional bias of the $k$-NN matching and local polynomial
estimators from $n^{-\ell/d}$ to $n^{-\ell/d_0}$. In fact, we can obtain similar results if $X$ and $X^*$ are not exactly supported on 
$\X_0$, but only up to some small noise. This means that the results for the rates of the final estimators extend to the case where the
distance from the covariate vector $X$ to the submanifold $\X_0$ is no larger, in distribution, than $n^{-1/d_0}$. In other words, we 
assume that we can write $X=X^{(0)}+n^{-1/d_0}\eta$, where $X^{(0)}$ is supported on $\X_0$, and $\eta$ is a random variable on $\R^d$
having a finite moment of order $d_0$. In that case, it is possible that the distribution of $X^{(0)}$ has intrinsic dimension $d_0$,
yet the distribution of $X$ does not satisfy the first part of Condition \ref{cond:A2} with $d_P=d_0$. However, the main results of this
paper are still valid with the intrinsic dimension $d_0$, because we again have, for $Q$-almost all $z\in\X$ and all $\la\in[0,d_0]$,
$$\E[\tak^\la]\le C\left(\E[\hat\tau_{k+1}^{(0)}(z)^\la]+n^{-\la/d_0}\right),$$ for some constant $C\in\Rs$, where $\hat\tau_{k+1}^{
(0)}(z)$ is the $(k+1)$th order statistic of the sample $(\norm{X_i^{(0)}-z})_{i\in\ieu}$. It implies that the conditional bias of the
final estimators has the rate $n^{-\ell/d_0}$.
In contrast, when the distance from $X$ to the submanifold $\X_0$ is greater than $n^{-1/d_0}$ in distribution, it means that the 
noise is too large and therefore adds more dimensions to the covariate vector $X$.

\section{Possible extensions}\label{extensions}

\textbf{Labels:} In the estimator $\hat h_0(z)$, the final labels $h(X_i,Y_i)$ can be replaced by $h(z,Y_i)$. This approach was considered
in \citet{portier2024scalable}. In this case, the assumptions on the regression function $g$ are replaced by assumptions on the function 
$x\mapsto\E[h(z,Y)\mid X=x]$ with a fixed $z$, where $x\in B(z,3r_0)$. The two approaches will have the exact same rates of convergence 
and will only slightly differ in the constants. The same applies to the local polynomial estimator $\hat h_L(z)$ in section 
\ref{polynomial}, where the final labels $h(X_i,Y_i)$ can be replaced by $h(z,Y_i)$ with only minor changes. Of course, when $h(x,y)$ is 
a function of $y$, as in the context of average treatment effects, the two approaches are strictly equivalent.
\bigskip

\textbf{Mixed distributions:} The results of this paper naturally extend to the case where the covariate vectors $X_1,\dots,X_n$ are
independent but not identically distributed, and similarly for the target vectors $X_1^*,\dots,X_m^*$. In this case, setting $P_i,\,
i\in\ieu$ the distribution of $X_i$, and $Q_j,\,j\in\ieu[m]$ the distribution of $X_j^*$, the distributions $P$ and $Q$ in Conditions
\ref{cond:A2} to \ref{cond:A5} now represent the mixed distributions $P:=\meann P_i$ and $Q:=\meanm Q_j$, respectively.
In particular, the first part of Condition \ref{cond:A2}, writing $P(B(z,r))\ge p(z)r^d$, translates the idea that sufficiently many
vectors $X_i,\,i\in\ieu$, have some probability to fall inside the ball $B(z,r)$.
\bigskip

\textbf{Optimal bandwidth:} In the low smoothness regime, that is when the regularity parameter $\ell$ is smaller than $d_P/2$, then
the $k$-NN matching and local polynomial estimators $\hat e_L(h)$ are not rate-optimal. Their rates of convergence are limited by 
their squared conditional bias term, which behaves like $n^{-2\ell/d_P}$, as soon as we have a good control of the transferability between
the source and target distributions. The pointwise estimators $\hat h(z)$ we gave are kernel estimators with the local bandwidth
$\tak$, which behaves like $n^{-1/d_P}$ when $k$ is fixed. Alternatively, we can use for $\hat h(z)$ a kernel estimator with the
non-adaptive local bandwidth $\tak^{2d_P/(2\ell+d_P)}$, which behaves like $n^{-2/(2\ell+d_P)}$, creating a very localised pointwise
estimator. In this case, the mean-squared error of the final estimator $\hat e(h)$ would have the rate $n^{-4\ell/(2\ell+d_P)}$, which
is an improvement over the rate $n^{-2\ell/d_P}$ when $\ell<d_P/2$.

\section{Numerical Experiments}\label{experiments}

In this section, we provide results for numerical experiments that show how the estimators perform with unbounded manifold supports, and
analyse how the transferability between the source and target distributions affects the convergence rates.
We compare the following methods: \begin{description}
\item[Matching:] The $1$-nearest neighbour matching estimator $\hat e_0(h)$ we considered.
\item[Sampling:] The $1$-nearest neighbour sampling estimator, that corresponds to the estimator $\hat e_0(h)$ where the final labels
$h(X_i,Y_i)$ are replaced by $h(z,X_i)$.
\item[Poly-$L$-M:] The $k$-NN local polynomial estimator $\hat e_L(h)$ we considered with $k:=2K^*(d,L)$. This parameter $k$ is lower than
the theoretical threshold $(2D+1)K^*$, as \citet{holzmann2026multivariate} reported that smaller $k$'s tended to result in better
performance. We include more neighbours that the minimal number $K^*$ for which the least-squares problem is well-posed for more
stability.
\item[Poly-$L$-S:] The $k$-NN local polynomial estimator where the final labels $h(X_i,Y_i)$ are replaced by $h(z,Y_i)$.
\item[KMM:] The \textit{Kernel Mean Matching method} \citep{huang2006correcting} ; \citep{gretton2009covariate} to weight source 
observations.
\item[KLIEP:] The \textit{Kullback-Leibler Importance Estimation Procedure method} \citep{sugiyama2007direct} to weight source 
observations.
\item[Oracle:] The hypothetical estimator $\meanm h(X_j^*,Y_j^*)$ with an oracle access to the target labels $Y_j^*$.
\end{description}

\bigskip

We conducted experiments in two setups described below. For each setup, the first $d_0$ coordinates of the covariate vector are 
independent realisations of a univariate distribution, and then the random vector is completed into a vector of dimension $d$ by applying
to the first coordinates functions such as $x\mapsto x^2$, $x\mapsto\cos(x)$, or $x\mapsto 1/(x^2+1)$.
\begin{description}
\item[Exponential-Sin] The covariate vector $X_{1:d_0}$ is made up of independent realisations of the exponential distribution 
$\mathcal E(1)$ for the target sample and $\mathcal E(\mu_P)$ for the source sample. The labels are generated according to the rule
$y=f(x)+\e$, where $f(x)=\sin(x_1)$, and $\e\sim\mathcal N(0,1)$ is independent from $x$. We define the function $h(x,y):=\cos(x_2)y+1$.
The value of the integral we want to estimate is $1.25$.
\item[Normal-Poly] The covariate vector $X_{1:d_0}$ is made up of independent realisations of the normal distribution $\mathcal N(0,1)$ for
the target sample and $\mathcal N(0,1/\mu_P)$ for the source sample. The labels are generated according to the rule $y=f(x)+\e$, where
$f(x)=\sum_{i=1}^{d_0}x_i^2$, and $\e\sim\mathcal N(0,1)$ is independent from $x$. We define the function $h(x,y):=y\sum_{i=1}^{d_0}x_i^2$.
The value of the integral we want to estimate is $d_0(d_0+2)/4$.
\end{description}
Note that the functions $f$ and $h$ we chose for the \textbf{Normal-Poly} setup are not bounded, but their growths are negligible 
compared to the decay of the source distribution $\mathcal N(0,1)$, reducing the integrability conditions in \ref{cond:A3} to transferability
conditions between the source and target measures only.

The first experiment illustrates the way the bias of the estimators behaves with different overlying dimensions $d$ and intrinsic
dimensions $d_0$. For this experiment, we fix $\mu_P=1/2$, and we let the sample sizes $m=n$ vary. The results are shown in Figures
\ref{fig:biasexp} and \ref{fig:biasnorm}. All methods have biases that get smaller as the sample sizes increase. The slopes for the
logarithmic bias depend on the methods, and they are steeper for the \textbf{Poly-$2$-M} method, as can be expected from the 
theoretical analysis. We can also see that while Subfigures \ref{fig:biasexp3v12}, \ref{fig:biasexp6v12} and \ref{fig:biasexp12v12}
show the same overlying dimension $d$, the rates in Subfigure \ref{fig:biasexp3v12} are closer to that of Subfigure 
\ref{fig:biasexp3v3}, which shares the same intrinsic dimension $d_0=3$, than to that of Subfigure \ref{fig:biasexp6v12} or 
\ref{fig:biasexp12v12}. In contrast, the biases of the estimators become larger as the intrinsic dimension increases, as the theory 
suggests. The same behaviour can be observed in the \textbf{Normal-Poly} setup in Figure \ref{fig:biasnorm}.

\begin{figure}
\centering\begin{subfigure}[t]{0.43\linewidth}\centering
\includegraphics[width=\linewidth]{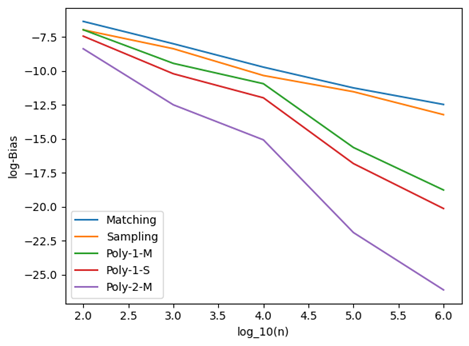}
\caption{Dimensions $d_0=3,\,d=3$}
\label{fig:biasexp3v3}
\end{subfigure}
\hfill\begin{subfigure}[t]{0.43\linewidth}\centering
\includegraphics[width=\linewidth]{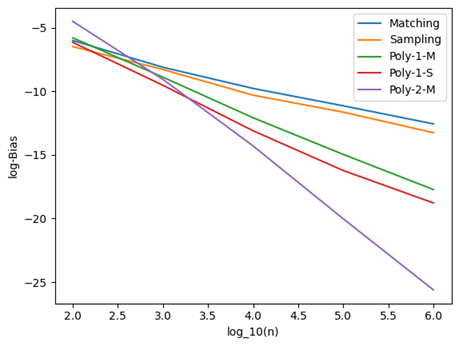}
\caption{Dimensions $d_0=3,\,d=12$}
\label{fig:biasexp3v12}
\end{subfigure}
\begin{subfigure}[t]{0.43\linewidth}\centering
\includegraphics[width=\linewidth]{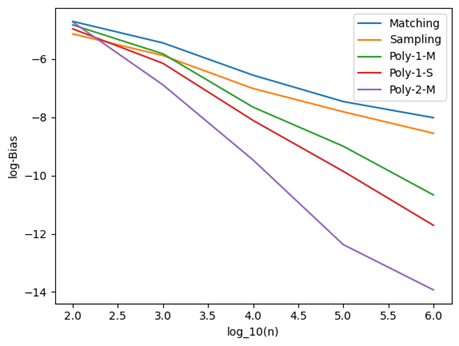}
\caption{Dimensions $d_0=6,\,d=12$}
\label{fig:biasexp6v12}
\end{subfigure}
\hfill\begin{subfigure}[t]{0.43\linewidth}\centering
\includegraphics[width=\linewidth]{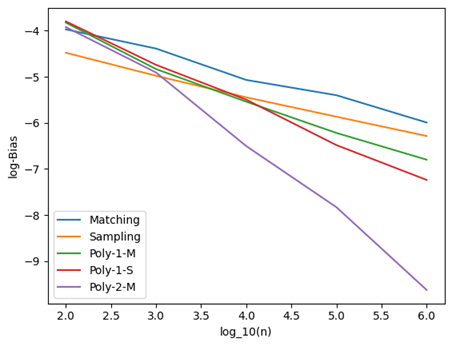}
\caption{Dimensions $d_0=12,\,d=12$}
\label{fig:biasexp12v12}
\end{subfigure}
\caption{Bias rates, Setup Exponential-Sin}
\label{fig:biasexp}
\end{figure}

\begin{figure}
\centering\begin{subfigure}[t]{0.43\linewidth}\centering
\includegraphics[width=\linewidth]{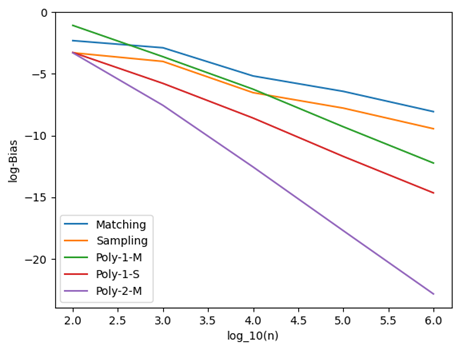}
\caption{Dimensions $d_0=3,\,d=3$}
\label{fig:biasnorm3v3}
\end{subfigure}
\hfill\begin{subfigure}[t]{0.43\linewidth}\centering
\includegraphics[width=\linewidth]{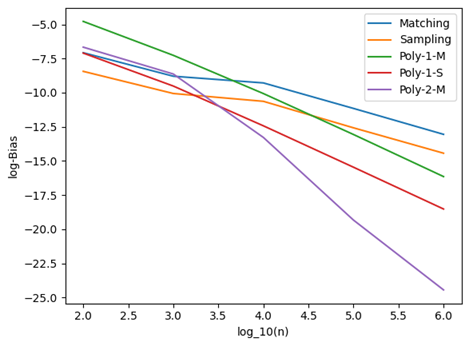}
\caption{Dimensions $d_0=3,\,d=12$}
\label{fig:biasnorm3v12}
\end{subfigure}
\begin{subfigure}[t]{0.43\linewidth}\centering
\includegraphics[width=\linewidth]{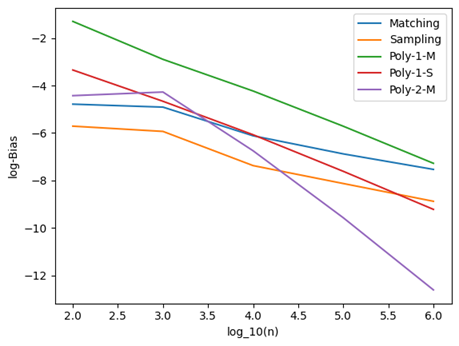}
\caption{Dimensions $d_0=6,\,d=12$}
\label{fig:biasnorm6v12}
\end{subfigure}
\hfill\begin{subfigure}[t]{0.43\linewidth}\centering
\includegraphics[width=\linewidth]{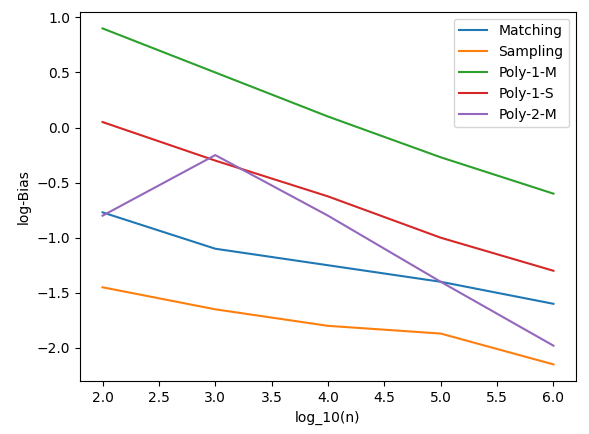}
\caption{Dimensions $d_0=12,\,d=12$}
\label{fig:biasnorm12v12}
\end{subfigure}
\caption{Bias rates, Setup Normal-Poly}
\label{fig:biasnorm}
\end{figure}

In the second experiment, we analyse the effect of the transferability between the source and the target measures on the performances of the
final estimators for the expectation. For this experiment, we fix the dimensions $d_0=d=2$ and the sample sizes $m=n=10^3$, and we let the
parameter $\mu_P$ for the source distribution vary. The results can be found in Figure \ref{fig:distor}. The source and target
distributions are the same when $\mu_P=1$, that is, for the value $0$ on the $x$-axis. In Figure \ref{fig:distorexpsin}, we can see that
all estimators have similar performances that vary little until the critical point $\mu_P=2$ (the value $1$ on the $x$-axis). Starting
from this value, only the ideal \textbf{Oracle} estimator is not affected by the transferability, since it does not depend on the source
sample. However, the performances of all the other estimators seem to decay after this critical point. A similar curve appears on the
right side of Figure \ref{fig:distornorlin}. It happens when $\mu_P\ge 2$, that is, when the condition $2=2\mu_Q>\mu_P$ is no longer
satisfied, as we have forecast in our analysis; see Propositions \ref{gaussian} and \ref{gamma}.

\begin{figure}
\centering\begin{subfigure}[t]{0.43\linewidth}\centering
\includegraphics[width=\linewidth]{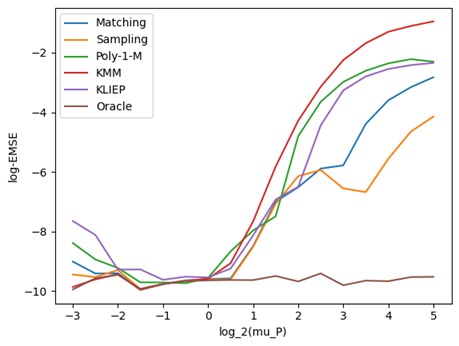}
\caption{Setup Exponential-Sin}
\label{fig:distorexpsin}
\end{subfigure}
\hfill\begin{subfigure}[t]{0.43\linewidth}\centering
\includegraphics[width=\linewidth]{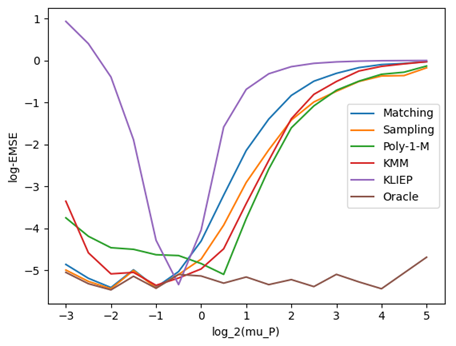}
\caption{Setup Normal-Poly}
\label{fig:distornorlin}
\end{subfigure}
\caption{Transferability and mean-squared errors}
\label{fig:distor}
\end{figure}

\section{Conclusion}\label{conclusion}

We showed that the $k$-NN matching and local polynomial estimators for covariate shift adaptation can be applied to covariate vectors 
with densities that can approach zero and infinity, and whose support can be an unbounded non-regular set. We provided optimal rates of 
convergence that depend on the intrinsic dimensions of the covariates and on the transferability between the source and target
probability measures. For future research, one could investigate how to generalise the results outside the framework of $k$-NN algorithms 
or study the optimisation of the censor radius.

\bigskip

\section{Proofs}\label{proofs}

In this section, we give the proofs for the main theorems in the paper. The proofs of Theorems \ref{simple_mse} and \ref{simple_poly_mse}
can be deduced from those of Theorems \ref{variance} and \ref{bias}, as simpler cases.

\subsection{Proof of Theorem \ref{minimax}}

We derive minimax lower bounds using Le Cam's method. Let $\Phi:\R\to\R_+$ denote the function defined by $\Phi(x):=C\exp(-x^2/2)$, where 
the constant $C\in\Rs$ is chosen small enough so that for all $s\in\ieu[\ell][0],\,\norm{D^s\Phi}_\infty\le s!$.
Now fix $\gamma\in[0,1]$ and $\Ccov\in\Rs$. We consider the function $g_1:\R^d\to\R_+$ defined by $g_1(z):=\Lambda\Phi(z_1-r)$, for a
constant $\Lambda\in\Rs$ and a location parameter $r\in\R$ to be chosen later. This regression function is compared with the regression function
$g_0:z\mapsto 0$. We set the source density at $f_P:z\mapsto e^{-z_1}$, and the target density at $f_Q:z\mapsto\alpha e^{-\alpha z_1}$, where
$\alpha:=\frac\gamma 2+\frac 1{\ln(n)}\in(0,1)$, both on the convex support $\X:=\R_+\times[0,1]^{d-1}$. The second and third points in the
definition of $\Sigma(\gamma,\Ccov)$ are satisfied. 
The transferability condition writes $$\Lambda^2\int_{\R^d}\frac{f_Q(z)^2}{f_P(z)^\gamma}\md z=\Lambda^2\alpha^2\int_0^\infty e^{-2z_1/\ln(n)}
\md z_1=(\Lambda\alpha)^2\frac{\ln(n)}2\le\Ccov,\text{ and}$$ 
$$\left(\Lambda\int_{\R^d}\frac{f_Q(z)}{f_P(z)^{\gamma/2}}\md z\right)^2=\Lambda^2\alpha^2\left(\int_0^\infty e^{-z_1/\ln(n)}\md z_1
\right)^2=(\Lambda\alpha\ln(n))^2=\Ccov,$$ setting $\Lambda^2:=(\alpha\ln(n))^{-2}\Ccov\le 1$ for $n$ large enough.
Next, we control the Kullback-Leibler divergence between the distributions $\Pb_1:=(P_{Y,1}\cdot P)\otimes Q$ and $\Pb_0:=(P_{Y,0}\cdot P)
\otimes Q$, corresponding to the regression functions $g_1$ and $g_0$, respectively. The only difference between $\Pb_1$ and $\Pb_0$ lies 
in the conditional distributions $P_{Y,1}$ and $P_{Y,0}$ given $X=z$, and these conditional distributions are Gaussian distributions with
the same variance $\Lambda^2$ and with respective locations $g_1(z)$ and $g_0(z)$. Therefore, we have
$$\operatorname{KL}(\Pb_1,\Pb_0)=\int_{\R^d}\frac{(g_1(z)-g_0(z))^2}{2\Lambda^2}f_P(z)\md z=\frac 12\int_0^\infty\Phi(z_1-r)^2
e^{-z_1}\md z_1\le\frac{e^{-r}}2P_\R(\Phi^2),$$ where $P_\R(\Phi^2):=\int_\R\Phi(x)^2e^{-x}\md x=C^2e^{1/4}\sqrt{2\pi}$. 
Setting the parameter $r$ such that $P_\R(\Phi^2)e^{-r}:=\frac{2\ln(2)}n$, we therefore have $\operatorname{KL}(\Pb_1,\Pb_0)\le\ln(2)/n$,
and applying Le Cam's theorem \citep{yu1997} yields 
$$\underset{\hat E_n}\inf\;\underset{(P_{X,Y},Q)\in\Sigma(\gamma,\Ccov)}\sup\E[(\hat E_n-e(g))^2]\ge\frac{(e(g_1)-e(g_0))^2}{32}.$$ However, 
$$e(g_1)-e(g_0)=\Lambda\int_\X g_1(z)f_Q(z)\md z=\Lambda\alpha e^{-\alpha r}\int_{-r}^\infty\Phi(x)e^{-\alpha x}\md x.$$ 
Since the function $\Phi$ is non-negative and non-increasing on $\R_+$, we have $$\alpha\int_{-r}^\infty\Phi(x)e^{-\alpha x}\md x\ge
\int_0^\infty\Phi(x)\,\alpha e^{-\alpha x}\md x\ge P(\Phi),$$ where $P(\Phi):=\int_{\R_+}\Phi(x)e^{-x}\md x>0$. As a consequence, we find
$$\underset{\hat E_n}\inf\;\underset{(P_{X,Y},Q)\in\Sigma(\gamma,\Ccov)}\sup\E[(\hat E_n-e(g))^2]\ge\frac{P(\Phi)^2\Lambda^2}{32}\,e^{-2
\alpha r}=\frac{P(\Phi)^2}{32}\Ccov(\alpha\ln(n))^{-2}\,e^{-2\alpha r},$$ hence the result, setting $\mathcal K:=\frac{e^{-2}P(\Phi)^2}{
32P_\R(\Phi^2)^2}$.

\subsection{Proof of Theorem \ref{variance}}

We divide the conditional variance into two parts, $V_{m,n}:=F_{m,n}+E_n$, where \begin{align*}
F_{m,n} & :=\hat e_L(h)-\E[\hat e_L(h)\mid (X_i,Y_i)_{i\in\ieu}],\text{ and} \\ E_n & :=\E[\hat e_L(h)\mid (X_i,Y_i)_{i\in\ieu}]-
\E[\hat e_L(h)\mid (X_i)_{i\in\ieu}].
\end{align*}
First, we have $F_{m,n}=\meanm\hat h_L(X_j^*)-\int_\X\hat h_L(z)\md Q(z)$, and therefore,
$$\E[F_{m,n}^2\mid (X_i,Y_i)_{i\in\ieu}]=\frac 1m\Var\left(\hat h_L(X^*)\mid (X_i,Y_i)_{i\in\ieu}\right).$$ As a consequence, 
$\E[F_{m,n}^2]\le\frac 1m\E\left[\hat h_L(X^*)^2\right]$. Applying Jensen's inequality yields \begin{align*}
\E[F_{m,n}^2] & \le\frac 1m\int_\X\E\left[e_0^\top M(z)^{-1}\sum_{i\in I_k(z)}(k\norm g_{\infty,B(z,r_0)\cap\X}^2+\s^2(z))\zeta(z,X_i)
\zeta(z,X_i)^\top M(z)^{-1}e_0^\top\1{\tak\le r_0}\right]\md Q(z) \\ & \le\frac 1m\int_\X\left(\norm g_{\infty,B(z,r_0)\cap\X}^2+\frac{
\s^2(z)}k\right)\E[ke_0^\top M(z)^{-1}e_0\1{\tak\le r_0}]\md Q(z) \\ & \le\frac{C_\X}m\int_\X\left(\norm g_{\infty,B(z,r_0)\cap\X}^2+
\frac{\s^2(z)}k\right)v(z)^{2DK^*}\md Q(z),\text{ applying Lemma }\ref{key}.
\end{align*}
For the second part of the conditional variance, we write $$E_n=\int_\X e_0^\top M(z)^{-1}\sum_{i\in I_k(z)}(h(X_i,Y_i)-g(X_i))
\zeta(z,X_i)\,\1{\tak\le r_0}\md Q(z),$$ and, therefore, \begin{multline*}
\E[E_n^2\mid (X_i)_{i\in\ieu}]=\int_{\X^2}e_0^\top M(y)^{-1}\sum_{i\in I_k(y)\cap I_k(z)}\Var(h(X_i,Y_i)\mid(X_j)_{i\in\ieu})\,
\zeta(y,X_i)\zeta(z,X_i)^\top \\ M(z)^{-1} e_0\,\1{\max\{\hat\tau_{k+1}(y);\tak\}\le r_0}\md Q^{\otimes 2}(y,z).    
\end{multline*}
Because $e_0^\top M(y)^{-1}\zeta(y,X_i)\zeta(z,X_i)M(z)^{-1}e_0\le(e_0^\top M(y)^{-1}\zeta(y,X_i))^2/2+(\zeta(z,X_i)^\top M(z)^{-1}
e_0)^2/2$, by the symmetry between $y$ and $z$, it yields \begin{multline*}
\E[E_n^2]\le\int_{\X^2}\min\{\s^2(y);\s^2(z)\}\E[e_0^\top M(z)^{-1}e_0\mathbf 1(\max\{\tak;\hat\tau_{k+1}(y)\}\le r_0, \\ \tak+\hat\tau_{
k+1}(y)\ge\norm{y-z})]\md Q^{\otimes 2}(y,z),\text{ and}
\end{multline*}
\begin{align*}
\E[E_n^2] & \le E_{n,1}+E_{n,2},\text{ where} \\ E_{n,1} & :=\int_\X\s^2(z)\,\E[\E[e_0^\top M(z)^{-1}e_0\mid\tak]\,Q(B(z,2\tak))\1{\tak
\le r_0}]\md Q(z),\text{ and} \\ E_{n,2} & :=\int_{\X^2}\s^2(y)\E[\E[e_0^\top M(z)^{-1}e_0\1{\tak\le\hat\tau_{k+1}(y)}\mid\hat\tau_{
k+1}(y)]\,\1{\hat\tau_{k+1}(y)\le r_0,z\in B(y,2\hat\tau_{k+1}(y))}\md Q^{\otimes 2}(y,z).
\end{align*}
Applying the second part of Condition \ref{cond:A2} and Lemma \ref{key}, we find the following. \begin{align*} 
E_{n,1} & \le\frac{C_\X}k\int_\X\s^2(z)\,v(z)^{2DK^*}\,\E[\tak^{d_Q-d_P}P(B(z,\tak))\1{\tak\le r_0}]\,\Rpq(z)\md Q(z) \\ & 
\le\frac{2C_\X(k+1)}{k(n+1)}r_0^{d_Q-\gres d_P}\int_\X\s^2(z)\,v(z)^{2DK^*}\left(\frac{(n+2)p(z)v(z)}{k+1}\right)^{1-
\gres}\Rpq(z)\md Q(z),
\end{align*}
where we used Lemma \ref{NN_radius} $(2)$. Similarly, applying Lemma \ref{key_cov} yields \begin{align*}
E_{n,2} & \le\frac{C_\X}k\int_\X\s^2(y)\,v(y)^{2DK^*}\E[Q(B(y,2\hat\tau_{k+1}(y)))\1{\hat\tau_{k+1}(y)\le r_0}]\md Q(y) \\ & \le
\frac{8C_\X}{k^{1-\gres}(n+1)^{\gres}}r_0^{d_Q-\gres d_P}\int_\X\s^2(y)\,v(y)^{2DK^*+1-
\gres}p(y)^{1-\gres}\Rpq(y)\md Q(y),
\end{align*}
where the integral that appears in both upper bounds on $E_{n,1}$ and $E_{n,2}$ is finite by the second part of Condition \ref{cond:A3}.

\subsection{Proof of Theorem \ref{bias}}

The conditional bias writes $B_n=\int_\X(\beta_n(z)-g(z)\1{\tak>r_0})\md Q(z)$, where $$\beta_n(z):=\meank(g(X_i)-g(z))\,\1{\tak\le 
r_0}.$$ We consider the decomposition $$\E[B_n^2]\le 2\left(\fbox 1+\fbox 2+\fbox 3+\fbox 4\right),\text{ where}$$ \begin{align*}
\fbox 1 & :=\E\left[\int_{\X^2}\beta_n(y)\beta_n(z)\1{\hat\tau_{k+1}(y)+\tak<\norm{y-z}}\md Q^{\otimes 2}(y,z)\right], \\
\fbox 2 & :=\E\left[\int_{\X^2}\beta_n(y)\beta_n(z)\1{\hat\tau_{k+1}(y)+\tak\ge\norm{y-z}}\md Q^{\otimes 2}(y,z)\right], \\
\fbox 3 & :=\E\left[\int_{\X^2}g(y)g(z)\,\1{\min\{\hat\tau_{k+1}(y);\tak\}>r_0,\,\norm{y-z}>2r_0}\md Q^{\otimes 2}(y,z)\right], \\
\fbox 4 & :=\E\left[\int_{\X^2}g(y)g(z)\1{\min\{\hat\tau_{k+1}(y);\tak\}>r_0,\,\norm{y-z}\le 2r_0}\md Q^{\otimes 2}(y,z)\right]. 
\end{align*}
We first deal with the low-distance case. In the double integral $\int_{\X^2}\beta_n(y)\beta_n(z)\md Q^{\otimes 2}(y,z)$, we distinguish 
between $\hat\tau_{k+1}(y)+\tak\ge\norm{y-z}$ and $\hat\tau_{k+1}(y)+\tak<\norm{y-z}$. In the latter case, using the conditional 
independence from Lemma \ref{cond_ind}, we find $$\fbox 1=\int_{\X^2}\E[\E[\beta_n(y)\mid\hat\tau_{k+1}(y)]\,\E[\beta_n(z)\mid\tak]\,
\1{\hat\tau_{k+1}(y)+\tak<\norm{y-z}}]\md Q^{\otimes 2}(y,z).$$ 

Applying Taylor's formula to the regression function $g$, we know that $g(z+t)=\zeta(z,z+t)^\top G(z)+R_L(z,z+t)$, where $G_\alpha(z):=\frac{
\partial_\alpha g(z)}{\alpha!},\,\alpha\in N_L$ and $$R_\ell(z,z+t):=\frac 1{(\ceil\ell-2)!}\int_0^1(1-s)^{\ceil\ell-2}(D^{\ceil\ell-1}
g(z+st)-D^{\ceil\ell-1}g(z))\left(t^{\otimes(\ceil\ell-1)}\right)\,\mathrm ds.$$
Then \begin{align*}
\hat G_L(z) & =\underset{\gamma\in\R^{N_L}}\argmin\sum_{i\in I_k(z)}[(G(z)-\gamma)^\top\zeta(z,X_i)+(h-g)(X_i,Y_i)+R_\ell(z,X_i)]^2 \\ & 
=G(z)+M^{-1}(z)\sum_{i\in I_k(z)}(R_\ell(z,X_i)+(h-g)(X_i,Y_i))\,\zeta(z,X_i). 
\end{align*}
Thus, we have \begin{align*}
\beta_n(z) & :=(\E[e_0^\top\hat G_L(z)\mid (X_i)_{i\in\ieu}]-g(z))\1{\tak\le r_0} \\ & =\sum_{i\in I_k(z)}R_\ell(z,X_i)\,e_0^\top 
M^{-1}(z)\zeta(z,X_i)\1{\tak\le r_0}.
\end{align*}
However, using the Cauchy-Schwarz inequality, we have \begin{align*}
\E[\beta_n(z)\mid\tak] & \le\E\left[ke_0^\top M(z)^{-1}\sum_{i\in I_k(z)}R_\ell(z,X_i)^2\zeta(z,X_i)\zeta(z,X_i)^\top M(z)^{-1}e_0\1{\tak
\le r_0}\mid\tak\right]^{1/2} \\ & \le\frac{\mathcal G_\ell(z)\tak^\ell}{\ceil{\ell-1}!}\,\E[ke_0^\top M(z)^{-1}e_0\mid\tak]^{1/2}
\1{\tak\le r_0} \\ & \le\frac{C_\X^{1/2}\,\mathcal G_\ell(z)\,v(z)^{DK^*}}{\ceil{\ell-1}!}\,\tak^\ell\,\1{\tak\le r_0},
\text{ applying Lemma }\ref{key}.
\end{align*}
The same applies to $\E[\beta_n(y)\mid\hat\tau_{k+1}(y)]$, so that applying the negative correlation from Lemma \ref{negative_correlation}
followed by Lemma \ref{NN_radius} (1) yields the following. \begin{align*}
\fbox 1 & \le\frac{C_\X}{(\ceil{\ell-1}!)^2}\left(\int_\X\mathcal G_\ell(z)\,v(z)^{DK^*}\E[\tak^\ell\mid\tak\le r_0]\md Q(z)\right)^2 \\ &
\le\frac{C_\X}{(\ceil{\ell-1}!)^2}\left(\frac{k+1}{n+1}\right)^{2\gbias}\left(2\Gamma(2+\floor{\gbias})
r_0^{\ell-\gbias d_P}\int_\X\mathcal G_\ell(z)\,v(z)^{DK^*}\,p(z)^{-\gbias}\md Q(z)\right)^2.
\end{align*}
The integral in the lase expression above is finite by the fourth part of Condition \ref{cond:A3}. Note that we replaced the
exponent $\ell$ by the lower exponent $d_P\gbias$ so that the last integrability condition in \ref{cond:A3} is less restrictive.

We now consider the case where the NN-balls intersect, that is, when $\hat\tau_{k+1}(y)+\tak\ge\norm{y-z}$. In this case, depending on 
whether $\hat\tau_{k+1}(y)\le\tak$ or $\hat\tau_{k+1}(y)>\tak$, we write
\begin{align*}
\fbox 2 & :=\E\left[\int_{\X^2}\beta_n(y)\beta_n(z)\1{\hat\tau_{k+1}(y)+\tak\ge\norm{y-z}}\md Q^{\otimes 2}(y,z)\right] \\ & \le
\int_{\X^2}\frac{\mathcal G_\ell(z)^2}{(\ceil{\ell-1}!)^2}\E[ke_0^\top M(z)^{-1}e_0\,\tak^{2\ell}\1{\max\{\tak;\hat\tau_{k+1}(y)\}\le
r_0,\tak+\hat\tau_{k+1}(y)\ge\norm{y-z}}]\md Q^{\otimes 2}(y,z) \\ & \le B_{n,1}+B_{n,2},
\end{align*}
where, first, \begin{align*}
B_{n,1} & :=\int_\X\frac{\mathcal G_\ell(z)^2}{(\ceil{\ell-1}!)^2}\E[ke_0^\top M(z)^{-1}e_0\,\tak^{2\ell}\,Q(B(z,2\tak))\,\1{\tak\le
r_0}]\md Q(z) \\ & \le\frac{C_\X}{(\ceil{\ell-1}!)^2}\int_\X\mathcal G_\ell(z)^2\,v(z)^{2DK^*}\E[\tak^{2\ell+d_Q-d_P}\,P(B(z,\tak))\,
\1{\tak\le r_0}]\Rpq(z)\md Q(z),
\end{align*}
using Lemma \ref{key} and the second part of Condition \ref{cond:A2}. Applying Lemma \ref{NN_radius} (1) or (2), we obtain
$$B_{n,1}\le\frac{2C_\X}{(\ceil{\ell-1}!)^2}\ceil{\gcov}!\,r_0^{2\ell+d_Q-\gcov d_P}\left(\frac{k+1}{n+1}
\right)^{\gcov}\int_\X\mathcal G_\ell(z)^2\,v(z)^{2DK^*+1}\,p(z)^{1-\gcov}\,\Rpq(z)\md Q(z),$$
where the last integral is finite by the third part of Condition \ref{cond:A3}. On the other hand, \begin{multline*}
B_{n,2}:=\int_{\X^2}\frac{\mathcal G_\ell(y)^2}{(\ceil{\ell-1}!)^2}\E[\E[ke_0^\top M(z)^{-1}e_0\1{\tak\le\hat\tau_{k+1}(y)}\mid\hat\tau_{
k+1}(y)]\hat\tau_{k+1}(y)^{2\ell} \\ \1{\hat\tau_{k+1}(y)\le r_0,z\in B(y,2\tau_{k+1}(y))}]\md Q^{\otimes 2}(y,z).
\end{multline*} 
Similarly, applying Lemma \ref{key_cov}, we find \begin{align*} 
B_{n,2} & \le\frac{C_\X}{(\ceil{\ell-1}!)^2}\int_\X\mathcal G_\ell(y)^2\,v(y)^{2DK^*}\E[\hat\tau_{k+1}(y)^{2\ell}Q(B(y,2\hat\tau_{
k+1}(y)))\1{\hat\tau_{k+1}(y)\le r_0}]\md Q(y) \\ & \le\frac{2C_\X}{(\ceil{\ell-1}!)^2}\ceil{\gcov}!\,r_0^{2\ell+d_Q-
\gcov d_P}\left(\frac{k+1}{n+1}\right)^{\gcov}\int_\X\mathcal G_\ell(y)^2\,v(y)^{2DK^*+1}\,p(y)^{1-
\gcov}\,\Rpq(y)\md Q(y),
\end{align*}
which gives the same term as for $B_{n,1}$.

We are now left with the high-distance conditional bias $\int_{\X^2}g(y)g(z)\1{\min\{\hat\tau_{k+1}(y);\tak\}>r_0}\md Q^{\otimes 2}(y,z)$. 
We distinguish this time between $\norm{y-z}\le 2r_0$ or $\norm{y-z}>2r_0$. In the latter case, we can apply Lemma 
\ref{negative_correlation} on negative correlation to obtain the following.\begin{align*}
\fbox 3 & \le\left(\int_\X\va{g(z)}\Pb(\tak>r_0)\md Q(z)\right)^2 \\ & \le\left(e^{1/4}\int_\X\va{g(z)}\exp\left(-\frac{(n+1)p(z)}{8
(k+1)}r_0^{d_P}\right)\md Q(z)\right)^2,\text{ applying Lemma }\ref{NN_radius}\,(3), \\ & \le\left(e^{1/4}(8r_0^{-d_P})^{\gbias}\left(
\frac{k+1}{n+1}\right)^{\gbias}\int_\X\va{g(z)}p(z)^{-\gbias}\md Q(z)\right)^2,
\end{align*}
using the inequality $e^{-x}\le x^{-\gbias}$ for all $x\in\Rs$. Finally, when $\norm{y-z}\le 2r_0$, we can write
\begin{align*}
\fbox 4 & \le 2\int_\X g(z)^2\,Q(B(z,2r_0))\,\Pb(\tak>r_0)\md Q(z) \\ & \le 2\int_\X g(z)^2\,P(B(z,r_0))\,\E[\tak^{d_Q-d_P}\1{\tak>r_0}]
\,\Rpq(z)\md Q(z),\text{ using Condition }\ref{cond:A2}, \\ & \le 2r_0^{d_Q-d_P}\int_\X g(z)^2\,\E[P(B(z,\tak))\1{\tak>r_0}]
\Rpq(z)\md Q(z) \\ & \le 2e^{1/4}r_0^{d_Q-d_P}\int_\X g(z)^2\,\frac{k+1}{n+1}\exp\left(-\frac{(n+2)p(z)}{8(k+1)}r_0^{d_P}
\right)\,\Rpq(z)\md Q(z),\text{ applying Lemma }\ref{NN_radius}\,(3), \\ & \le 2e^{1/4}r_0^{d_Q-\gcov d_P}
8^{\gcov-1}\left(\frac{k+1}{n+1}\right)^{\gcov}\int_\X g(z)^2\,p(z)^{1-\gcov}\,\Rpq(z)\md Q(z),
\end{align*}
where the integral is finite by the third part of Condition \ref{cond:A3}.

\subsection{Proof of Theorem \ref{better_bias}}

The high-distance case is dealt with exactly the same way as in the proof of Theorem \ref{bias}. Similarly, the proof for the quantity $\fbox 2$
in Theorem \ref{bias} still holds, replacing the regularity parameter $\ell$ by $1$. For the main quantity $\fbox 1$, we now apply a 
Taylor expansion on $g$. We have for all $i\in I_k(z)$, $$\va{g(X_i)-g(z)-\nabla g(z)^\top(X_i-z)}\le\mathcal G_\ell(z)\norm{X_i-
z}^\ell.$$ But conditionally on $\hat\tau_{k+1}(z)$ and $I_k(z)$, the random variables $(X_i)_{i\in I_k(z)}$ are independent and
identically distributed, and their common distribution is the trace of $P$ on the open ball $B^\circ(z,\tak)$, thanks to the second part
of Lemma \ref{cond_ind}. Because we assumed that for $Q$-almost all $z\in\X,\,P$ does not charge the spheres of center $z$ and radius
$r\in[0,r_0]$, the common conditional distribution is also the trace of $P$ on the closed ball $B(z,\tak)$. Therefore, using the notation
from Condition \ref{cond:A4}, if $\tak\le\delta(z)$, then for all $i\in I_k(z)$, almost-surely, \begin{align*}
\E[X_i-z\mid I_k(z),\tak] & =\frac 1{P(B(z,\tak))}\int_{B(z,\tak)}(x-z)\md P(x) \\ & =\frac 1{2P(B(z,\tak))}\int_{B(0,\tak)}x\md P_z(x),
\\ \norm{\E[X_i-z\mid I_k(z),\tak]} & \le\frac\tak 2\frac 1{P(B(z,\tak))}\va{P_z}(B(0,\tak)) \\ & \le\Lambda_P(z)\tak^\ell,
\text{ using Condition }\ref{cond:A4}.
\end{align*}
It implies that $\va{\E[\beta_n(z)\mid\tak]}\1{\tak\le\delta(z)}\le(\mathcal G_\ell(z)+\norm{\nabla g(z)}\Lambda_P(z))\,\tak^\ell$, 
while on the other hand $\va{\E[\beta_n(z)\mid\tak]}\1{\delta(z)\le\tak\le r_0}\le\norm{\nabla g(z)}\tak$.
Because these two upper bounds are non-decreasing functions of $\tak$ times the indicator $\1{\tak\le r_0}$, we can apply Lemma
\ref{negative_correlation} on negative correlation. The remaining computations for the quantity where $\tak\le\delta(z)$ are the same as 
in the proof of Theorem \ref{bias}. For the quantity where $\delta(z)<\tak$, we first apply Lemma \ref{NN_radius} (4) to obtain 
$$\E[\tak\1{\delta(z)<\tak}\mid\tak\le r_0]\le 2e^{1/4}\left(\frac{k+1}{(n+1)p(z)}\right)^{1/d_P}\exp\left(-\frac{(n+1)p(z)}{8(k+1)}
\delta(z)^{d_P}\right),$$ and then we are in a position to use Condition \ref{cond:A5} with $L:=\frac{n+1}{k+1}$, which completes the proof.

\subsection{Proof of Theorem \ref{ate_poly}}

We first check the validity of Conditions \ref{cond:A1} to \ref{cond:A5}. It is sufficient to check the volume bounding
conditions and the integrability conditions. For $w\in\{0;1\}$, set $S_w$ the conditional distribution of $X$ given $W=w$. Then using
Condition \ref{cond:A2'}, we have, for $S$-almost all $x\in\X$ and all $r\in[0,2r_0]$, $$\eta_w(x)p(x)r^{d_0}\le S_w(B(x,r))\le
(1-\eta_{1-w}(x))p(x)v(x)r^{d_0},\quad\text{and}$$ $$S_{1-w}(B(x,2r))\le 2^{d_0}\frac{v(x)(1-\eta_w(x))}{\eta_w(x)}S_w(B(x,r)).$$
Using the functions $x\mapsto\eta_w(x)p(x)$ and $x\mapsto v(x)(1-\eta_w(x))/\eta_w(x)$, the integrability conditions that appear
in Condition \ref{cond:A3} then match the ones from Condition \ref{cond:A3'}. We are then in a position to apply Theorems \ref{variance} 
and \ref{bias} twice, first by setting $P$ the conditional distribution $S_0$ of $X$ given $W=0$ and $Q$ the conditional
distribution $S_1$ of $X$ given $W=1$, and second by inverting the roles of $W=0$ and $W=1$. We then obtain conditional rates of
convergence given the random variable $W$, depending on the conditional sample sizes, that is, on the number of observations for each
treatment $N_w:=\Card(\{i\in\ieu[N]\mid W_i=w\}),\,w\in\{0;1\}$. However, because for all $w\in\{0;1\}$, we have $$0<\int_\X\eta_w(x)
\md S(x)\le\Pb(W=w)\le\int_\X(1-\eta_{1-w}(x))\md S(x)<1,$$ we know that for all $a\in\Rs,\,\E[N^aN_w^{-a}\1{N_w>0}]$ remains bounded as
$N$ goes to infinity, as shown in \citet{viel2025convergence}.

\subsection{Proof of Proposition \ref{gaussian}}

We first check the volume bounding conditions, choosing for $\norm\cdot$ the Euclidean norm on $\R^d$. Set $\la_P$ and $\la_Q$ the largest
eigenvalues of $\Sigma_P^{-1}$ and $\Sigma_Q^{-1}$, respectively. Let $z\in\R^d$, if $\norm{z-\mu_P}\le r_0$, then we can lower bound $f_P$ 
on $B(z,r_0)$ be a positive constant, for instance $\frac 1{\sqrt{2\pi\det(\Sigma_P)}}\exp(-2\la_Pr_0^2)=\underset{B(\mu_P,2r_0)}\inf f_P>0$.

If $\norm{z-\mu_P}>r_0$, then the inequality $f_P\ge f_P(z)$ holds on an intersection of halfspaces in $B(z,r_0)$, with volume 
$2^{-d}\va{B(0,r_0)}$. It means that the first part of Condition \ref{cond:A2} is satisfied with the function
$$p:z\mapsto f_P(z)\,\vab\,(\exp(-2\la_Pr_0^2)\wedge 2^{-d}),$$ which is the density function $f_P$ up to some constant.
Similarly, the maximum of $f_P$ over $B(z,3r_0)$ is upper bounded by $f_P(z)\exp(6\la_Pr_0\norm z)$, so that we can set $v:z\mapsto\exp(6\la_P
r_0\norm z)$ up to some positive constant. The same can be applied to the target density $f_Q$, so that we can set $\Rpq:z\mapsto\frac{f_Q(z)}{
f_P(z)}\exp(4\la_Qr_0\norm z)$. 
We now discuss the integrability conditions in \ref{cond:A3}. The first integrability condition is trivially satisfied, for all choices
of parameters. With the assumptions on $\s^2$ and $g$, and with the choice of the multiplicative amplitude $v$ above, the second and third
integrability conditions are equivalent to the integrability of the function $z\mapsto\frac{f_Q(z)^2}{f_P(z)^\gamma}\exp(B\norm z)$ on $\R^d$,
for some constant $B\in\R$. The latter function is integrable if and only if the function $z\mapsto\exp\left(-\frac 12z^\top(2\Sigma_Q^{-1}-
\gamma\Sigma_P^{-1})z\right)$ is integrable, which happens if and only if the symmetric matrix $2\Sigma_Q^{-1}-\gamma\Sigma_P^{-1}$ is 
positive semi-definite. The same situation occurs for the fourth integrability condition.
Next, the local log-Lipschitz constant $\Lambda_P(z)$ can be taken as the supremum of the norm of the function $\nabla\ln(f_P)$ over
$B(z,r_0)$, which is proportional to the $\norm z$. It ensures the finiteness of the integral $\Cbias'$ under the same condition as for the
integral $\Cbias$. Finally, note that Condition \ref{cond:A5} is automatically satisfied as the density function $f_P$ is regular on the whole
space $\R^d$, and therefore there are no boundary issues. More precisely, we have set $\delta(z):=r_0$ for all $z\in\R^d$.

\bigskip

The proofs of Propositions \ref{gamma} and \ref{pareto} are skipped as they follow the same pattern as for Gaussian distributions.

\subsection{Proof of Proposition \ref{uniform}}

We first discuss the volume bounding conditions, choosing for $\norm\cdot$ the supremum norm on $\R^d$. Set $\alpha:=s(d-1)$ and $\beta:=
\alpha-d+1=(s-1)(d-1)\ge 0$. Below, we will write $\phi(z,r)\sim\psi(z,r)$ whenever there exist positives constants $c<C$ not depending on $z$
and $r$ such that $c\psi(z,r)\le\phi(z,r)\le C\psi(z,r)$.
Let $z\in\X$ and $r\in[0,r_0]$, then a direct computation gives $$P(B(z,r))\sim\left\{\begin{array}{ccc}z_1^{\mu_P}r^d & \text{ if } & r\le 
z_1^s \\ z_1^{\mu_P+\alpha}r & \text{ if } & z_1^s\le r\le z_1 \\ z_1^{\mu_P}r^{\alpha+1} & \text{ if } & z_1\le r\end{array}\right..$$
As a consequence, the first part of Condition \ref{cond:A2} is satisfied with the function $p(z)\sim z_1^{\mu_P+\beta}$. The same computations
show that we can also set $\Rpq(z)\sim f_Q(z)/f_P(z)\sim z_1^{\mu_Q-\mu_P}$ and $v(z)\sim 1$.
We now discuss the integrability conditions in \ref{cond:A3}. The first integrability condition is trivially satisfied. For the second and
third integrability conditions, the function $\Rpq\,p^{1-\gamma}\,f_Q$ is integrable over $\X$ if and only if the function $z_1\mapsto z_1^{
\alpha+2\mu_Q-\mu_P+(1-\gamma)(\mu_P+\beta)}$ is integrable when $z_1\to 0$, which occurs if and only if $2\mu_Q-\gamma\mu_P+\alpha+(1-\gamma)
\beta>-1$. Similarly, we can show that the function $p^{-\gamma/2}f_Q$ is integrable over $\X$ if and only if $\alpha+\mu_Q-\gamma(\mu_P+
\beta)/2>-1$. We find the system $$\left\{\begin{array}{ccc}2\mu_Q-\gamma\mu_P+(2s-1)(d-1)-\gamma\beta & > & -1 \\ 2\mu_Q-\gamma\mu_P+2s(d-1)-
\gamma\beta & > & -2\end{array}\right.,$$ where the second line is redundant, and therefore, the system holds if and only if we have the
inequality $2\mu_Q+((2-\gamma)s-1+\gamma)(d-1)>\gamma\mu_P-1$, as we claimed.

\bigskip

\bibliographystyle{plainnat} 
\bibliography{references}

\appendix

\begin{lemma}[Moments of the NN-radius]\label{NN_radius}
Suppose that we have the first and third parts of Condition \ref{cond:A2}, then for $Q$-almost all $z\in\X$, all $\la\in\Rs,\,\mu\in(0,d_P)$,
and all $a\in(0,r_0)$, \begin{align*}
\E[\tak^\la\mid\tak\le r_0] & \le C_{\la,d_P}\left(\frac{k+1}{(n+1)p(z)}\right)^{\la/d_P}, \\ \E[\tak^{-\mu}\1{\tak\le r_0}] & \le 2
\left(\frac{(n+1)p(z)v(z)}k\right)^{\mu/d_P} \\ \Pb(\tak>r_0) & \le e^{1/4}\exp\left(-\frac{(n+1)p(z)}{8(k+1)}r_0^{d_P}\right), \\ 
\E[\tak^\la\1{a<\tak}\mid\tak\le r_0] & \le e^{1/4}C_{\la,d_P}\left(\frac{k+1}{(n+1)p(z)}\right)^{\la/d_P}\exp\left(-\frac{(n+1)p(z)}{
8(k+1)}a^{d_P}\right),
\end{align*}
with the constant $C_{\la,d_P}:=2\Gamma(2+\floor{\la/d_P})$.
\end{lemma}
\begin{proof}
The proof of the first point is very similar to \citet{viel2025convergence}, Lemma 4. More precisely, let $F_z:r\mapsto P(B(z,r)),\,
F_z^{-1}:u\mapsto\inf\{r\in\R\,:\,F_z(r)\geq u\}$ be its generalised inverse, and let $(U_1,\dots,U_n)$ be an $n$-sample of the uniform
distribution on $[0,1]$. Then the two random vectors $(\norm{X_1-z},\dots,\norm{X_n-z})$ and $(F_z^{-1}(U_1),\dots,F_z^{-1}(U_n))$ have
the same distribution. Because the function $F_z^{-1}$ is non-decreasing, the two vectors $(\hat\tau_1(z),\dots,\hat\tau_n(z))$ and
$(F_z^{-1}(U_{(1)}),\dots,F_z^{-1}(U_{(n)}))$ also have the same distribution. In particular, defining the incomplete beta function
$\beta(s,t;u_0):=\int_0^{u_0}u^{s-1}(1-u)^{t-1}\md u$, we can write $$\E[\tak^\la\mid\tak\le r_0]=\frac 1{\beta(k+1,n-k;F_z(r_0))}
\int_0^{F_z(r_0)}(F_z^{-1}(u))^\la\,\1{F_z^{-1}(u)\le r_0}\,u^k(1-u)^{n-k-1}\,\mathrm du.$$

However, using the first part of Condition \ref{cond:A2}, we know that for all $r\in[0,r_0]$, we have $F_z(r)\geq p(z)r^{d_P}$, so that
for all $u\in[0,F_z(r_0)]$, $F_z^{-1}(u)\le\left(\frac u{p(z)}\right)^{1/d_P}$. It means that $$\E[\tak^\la\mid\tak\le r_0]\le\frac 1{
\beta(k+1,n-k;F_z(r_0))}\,p(z)^{-\la/d_P}\beta(k+1+\la/d_P,n-k;F_z(r_0)).$$ For all $s'\ge s\ge 1$, $u_0^{s-1}(1-u_0)^{t-1}\int_0^{u_0}
u^{s'-1}(1-u)^{t-1}\md u\le u_0^{s'-1}(1-u_0)^{t-1}\int_0^{u_0}u^{s-1}(1-u)^{t-1}\md u$, and therefore the function $u_0\mapsto\beta(
s',t;u_0)/\beta(s,t;u_0)$ is non-decreasing as its derivative is nonnegative. It yields \begin{align*}
\E[\tak^\la\mid\tak\le r_0] & \le p(z)^{-\la/d_P}\,\frac{\beta(k+1+\la/d_P,n-k)}{\beta(k+1,n-k)} \\ & =p(z)^{-\la/d_P}\frac{\Gamma(n+1)}{
\Gamma(n+\la/d_P+1)}\frac{\Gamma(k+1+\la/d_P)}{\Gamma(k+1)} \\ & \le p(z)^{-\la/d_P}\frac 2{(n+1)^{\la/d_P}}\,\Gamma(2+\floor{\la/d_P})
(k+1)^{\la/d_P},
\end{align*}
using the inequalities $\Gamma(a+s)/\Gamma(a)\le\Gamma(2+\floor s)\,a^s$ and $\Gamma(a)/\Gamma(a+s)\le 2\,a^{-s}$, as shown in
\citet{portier2024scalable}. The second point follows from the same reasoning, with the inequality $F_z^{-1}(u)\ge\left(\frac u{p(z)v(z)}
\right)^{1/d_P}$, for all $u\in[0,F_z(r_0)]$. Similarly, the proof of the last two points can be obtained by minor modifications of
\citet{viel2025convergence}, Lemmas 5 and 6. Details are omitted.
\end{proof}

\begin{lemma}[Conditional independence]\label{cond_ind}
Let $y,z\in\X$ and let $\Psi_y,\Psi_z$ be $P^{\otimes k}$-integrable functions invariant under permutations of the coordinates, defined 
on the respective NN-balls $B(y,\hat\tau_{k+1}(y))$ and $B(z,\tak)$. For $a\in\{y,z\}$, set $D_k(a):=(X_i)_{i\in I_k(a)}$, then on the
event $(\hat\tau_{k+1}(y)+\tak<\norm{y-z})$ we have the conditional independence, almost-surely, 
$$\E[\Psi_y(D_k(y))\Psi_z(D_k(z))\mid\hat\tau_{k+1}(y),\tak]=\E[\Psi_y(D_n(y))\mid\hat\tau_{k+1}(y)]\,\E[\Psi_z(D_k(z))\mid\tak],$$ 
and the conditional distribution, almost-surely, 
$$\E[\Psi_z(D_k(z))\mid\tak]=\frac 1{P(B(z,\tak))^k}\int_{B(z,\tak)^k}\Psi_z(\bar x)\md P^{\otimes k}(\bar x).$$
\end{lemma}
\begin{proof}
The event $(\hat\tau_{k+1}(y)+\tak)<\norm{y-z}$ occurs if and only if there exists a unique tuple $(i_y,i_z,J_y,J_z)$, where $i_y\neq i_z
\in\ieu$, $J_y,J_z$ are subsets of $\ieu\backslash\{i_y;i_z\}$ of cardinal $k$ with $J_y\cap J_z=\emptyset$, such that the event 
$\mathcal A_{J_y,J_z,i_y,i_z}(X_1,\dots,X_n)$, described below, occurs: \begin{itemize}
\item for $a\in\{y;z\}$ and all $j\in J_a,\,\norm{X_j-a}<\norm{X_{i_a}-a}$,
\item for $a\in\{y;z\}$ and all $j\in\ieu\backslash J_y\sqcup J_z\sqcup\{i_y;i_z\},\,\norm{X_j-a}>\norm{X_{i_a}-a}$,
\item $\norm{X_{i_y}-y}+\norm{X_{i_z}-z}<\norm{y-z}$.
\end{itemize}
If $F:\R_+^2\to\R$ is a measurable bounded function, we have \begin{multline*}
\E[\Psi_y(D_k(y))\Psi_z(D_k(z))F(\hat\tau_{k+1}(y),\tak)\1{\hat\tau_{k+1}(y)+\tak<\norm{y-z}}]= \\ \mathcal K_n\int_{\X^n}\Psi_y(x_1,
\dots,x_k)\Psi_z(x_{k+1},\dots,x_{2k})F(\norm{x_{2k+1}-y},\norm{x_{2k+2}-z}) \\ \1{\mathcal A_{\ieu[k],\ieu[2k][k+1],2k+1,2k+2}(x_1,
\dots,x_n)}\md P^{\otimes n}(x_1,\dots,x_n),  
\end{multline*}
where $\mathcal K_n:=n(n-1)\binom{n-2}k\binom{n-2-k}k$, hence the conditional independence and distribution after normalisation.
\end{proof}

\begin{lemma}[Negative Correlation]\label{negative_correlation}
Let $y,z\in\X$ and $r_y,r_z\in[0,r_0]$ be such that $\norm{y-z}>r_y+r_z$. Then we have the negative correlation
$$\Pb(\hat\tau_{k+1}(y)>r_y,\tak>r_z)\le\Pb(\hat\tau_{k+1}(y)>r_y\mid\hat\tau_{k+1}(y)\le r_0)\,\Pb(\tak>r_z\mid\tak\le r_0).$$
Let $y,z\in\X$ and let $\Phi_y,\Phi_z:\R_+\to\R_+$ be non-decreasing functions, then we have the negative correlation \begin{multline*}
\E[\Phi_y(\hat\tau_{k+1}(y))\Phi_z(\tak)\1{\hat\tau_{k+1}(y)+\tak<\norm{y-z},\max\{\hat\tau_{k+1}(y);\tak\}\le r_0}]\le \\ 
\E[\Phi_y(\hat\tau_{k+1}(y))\mid\hat\tau_{k+1}(y)\le r_0]\,\E[\Phi_z(\tak)\mid\tak\le r_0]    
\end{multline*}
\end{lemma}
\begin{proof}
We prove the first result by induction on the number of observations $n$. The result is true when $n=k+1$. We define the probabilities
\begin{align*}
I_{\ell,\ell'}^{(n)} & :=\Pb\left(\sum_{i=1}^n\mathbf 1(X_i\in B(y,r_y))\leq\ell,\,\sum_{i=1}^n\mathbf 1(X_i\in B(z,r_z))\le\ell'\right),
\\ A_\ell^{(n)} & :=\Pb\left(\sum_{i=1}^n\mathbf 1(X_i\in B(y,r_y))\leq\ell\mid\hat\tau_{k+1}(y)\le r_0\right),\quad\text{and} \\ 
B_{\ell'}^{(n)} & :=\Pb\left(\sum_{i=1}^n\mathbf 1(X_i\in B(z,r_z))\leq\ell\mid\tak\le r_0\right),
\end{align*}
as well as the two quantities $a:=\Pb(X_n\in B(y,r_y))$ and $b:=\Pb(X_n\in B(z,r_z))$. To apply our induction hypothesis, we apply the 
law of total probability to the random variable $X_n$, which produces the equations $$I_{\ell,\ell'}^{(n)}=aI_{\ell-1,\ell'}^{(n-1)}+b
I_{\ell,\ell'-1}^{(n-1)}+(1-a-b)I_{\ell,\ell'}^{(n-1)}.$$ We also have the conditional negative correlations \begin{align*}
A_\ell^{(n)} & \ge aA_{\ell-1}^{(n-1)}+(1-a)A_\ell^{(n-1)},\quad\text{and} \\ B_{\ell'}^{(n)} & \ge bB_{\ell'-1}^{(n-1)}+(1-b)B_{\ell'}^{
(n-1)}.
\end{align*}
Therefore, \begin{multline*}
A_\ell^{(n)}B_{\ell'}^{(n)}\ge abA_{\ell-1}^{(n-1)}B_{\ell'-1}^{(n-1)}+a(1-b)A_{\ell-1}^{(n-1)}B_{\ell'}^{(n-1)}+b(1-a)A_\ell^{(n-1)}
B_{\ell'-1}^{(n-1)} \\ +(1-a)(1-b)A_\ell^{(n-1)}B_{\ell'}^{(n-1)}.    
\end{multline*}
We recognise the terms $aA_{\ell-1}^{(n-1)}B_{\ell'}^{(n-1)},\,bA_\ell^{(n-1)}B_{\ell'-1}^{(n-1)}$, and $(1-a-b)A_\ell^{(n-1)}B_{\ell'}^{
(n-1)}$, so that applying the induction hypothesis, we find \begin{align*}
A_\ell^{(n)}B_{\ell'}^{(n)}-I_{\ell,\ell'}^{(n)} & \geq ab\,(A_{\ell-1}^{(n-1)}B_{\ell'-1}^{(n-1)}-A_{\ell-1}^{(n-1)}B_{\ell'}^{(n-1)}-
A_\ell^{(n-1)}B_{\ell'-1}^{(n-1)}+A_\ell^{(n-1)}B_{\ell'}^{(n-1)}) \\ & =ab(A_\ell^{(n-1)}-A_{\ell-1}^{(n-1)})(B_{\ell'}^{(n-1)}-
B_{\ell'-1}^{(n-1)}) \\ & \geq 0,   
\end{align*}
which concludes the proof by induction.

For the second point, we can write the expectation of non-negative random variables thanks to tail distribution functions as follows. Let
$\Phi^+$ denote the right-continuous generalised inverse of a non-decreasing function $\Phi$, then \begin{align*}
\fbox{1-1} & :=\E[\Phi_y(\hat\tau_{k+1}(y))\,\Phi_z(\tak)\,\1{\hat\tau_{k+1}(y)+\tak<\norm{y-x},\max\{\hat\tau_{k+1}(y);\tak\}\le r_0}] 
\\ &  =\int_{\R_+^2}\Pb(\Phi_y(\hat\tau_{k+1}(y))>s,\,\Phi_z(\tak)>t,\,\hat\tau_{k+1}(y)+\tak<\norm{y-x},\max\{\hat\tau_{k+1}(y);\tak\}
\le r_0)\,\mathrm d(t,s) \\ & \leq\int_{\R_+^2}\1{\max\{\Phi_y^+(t);\Phi_z^+(s)\}\le r_0,\Phi_y^+(t)+\Phi_z^+(s)<\norm{y-x}}\Pb(\hat\tau_{
k+1}(y)>\Phi_y^+(s),\,\tak>\Phi_z^+(t))\,\mathrm d(t,s) \\ & \leq\int_{\R_+^2}\Pb(\tak>\Phi_z^+(t)\mid\tak\le r_0)\,\Pb(\hat\tau_{k+1}(y)>
\Phi_y^+(s)\mid\hat\tau_{k+1}(y)\le r_0)\,\mathrm d(t,s) \\ & =\E[\Phi_z(\tak)\mid\tak\le r_0]\,\E[\Phi_y(\hat\tau_{k+1}(y))\mid
\hat\tau_{k+1}(y)\le r_0].   
\end{align*}
\end{proof}

\begin{lemma}[Key Lemma]\label{key}
With the notation from Section \ref{polynomial}, under Condition \ref{cond:A2}, and assuming that $k\ge(2D+1)K^*$, 
there exists a constant $C_\X$ depending on $L$ and $d$, such that for $Q$-almost all $z\in\X$, almost-surely,
$$\E[e_0^\top M(z)^{-1}e_0\mid\tak]\,\1{\tak\le r_0}\le\frac{C_\X}k\,v(z)^{2DK^*}.$$
\end{lemma}
\begin{proof}
We use the proof technique of \citet{holzmann2026multivariate}, following the structure and notation from \citet{viel2025convergence}, 
Lemma $1$. The only difference from our setup is the use of Conditions \ref{cond:A2} along with the indicator $\1{\tak\le r_0}$ instead of
the strong density assumption. This leads to the upper bound $$\Pb\left(\va{\det(\Xi(V))}\le\e\mid\tak\right)\le\sum_{i=1}^D\va{
\mathcal C_i(\e)}\left(\frac{p(z)v(z)\tak^d}{P(B(z,\tak))}\right)^{K^*}\le Cv(z)^{K^*}\e^{1/D},$$ for some constant $C>0$. 
We can then resume the proof to obtain the result stated above.
\end{proof}

\begin{lemma}[Covariance Key Lemma]\label{key_cov}
With the notation from Section \ref{polynomial}, under Condition \ref{cond:A2}, and assuming that $k\ge(2D+1)K^*$, 
there exists a constant $C_\X$ depending on $L$ and $d$, such that for $Q^{\otimes 2}$-almost all $(y,z)\in\X^2$, almost-surely,
$$\E[e_0^\top M(z)^{-1}e_0\1{\tak\le\hat\tau_{k+1}(y)}\mid\hat\tau_{k+1}(y)]\,\1{\hat\tau_{k+1}(y)\le r_0,z\in B(y,2\hat\tau_{k+1}(y))}
\le\frac{C_\X}k\,v(y)^{2DK^*}.$$
\end{lemma}
\begin{proof}
Keeping the same notation as above, we apply this time the last part of Condition \ref{cond:A2} to the Borel set $z+\tak\mathcal C_i(\e)$,
which, because $\tak\le\hat\tau_{k+1}(y)$ and $\norm{y-z}\le 2\hat\tau_{k+1}(y)$, is contained in some set $y+\hat\tau_{k+1}(y)
\mathcal A_i(\e)$, where $\mathcal A_i(\e)\subset B(0,3)^{K^*}$ also satisfies $\va{\mathcal A_i(\e)}\le\mathcal K\e^{1/D}$ for some 
constant $\mathcal K>0$. It yields $$\Pb\left(\va{\det(\Xi(V))}\le\e\mid\hat\tau_{k+1}(y)\right)\le Cv(y)^{K^*}\e^{1/D},$$ for some constant
$C>0$. The rest of the proof is unchanged.
\end{proof}

\end{document}